\documentclass[ssy]{informs3_no_remark}
\OneAndAHalfSpacedXI % current default line spacing
\usepackage{natbib}
 \bibpunct[, ]{(}{)}{,}{a}{}{,}%
 \def\bibfont{\small}%

 \usepackage{tikz,calc}
 \usetikzlibrary{decorations.markings}
 \usepackage{soul}
 \setul{}{0.75pt}
 \setuloverlap{1.5pt}

\usepackage{subfigure}
\usepackage{mwe}
\usepackage{lscape,booktabs,longtable}
\usepackage{amsmath}
\usepackage[pdftex,colorlinks=true,urlcolor=blue,citecolor=blue,pdfstartview=FitH]{hyperref}
\usepackage{comment}
\usepackage{enumerate}
\usepackage[ruled,vlined]{algorithm2e}
\usepackage[graphicx]{realboxes}

\newcommand{\E}[1]{{\mathrm{E}\left[#1\right]}}

\newcommand{\PP}[1]{\mathrm{P}\left( #1 \right)}
\newcommand{\Var}[1]{{\mathrm{Var}\left(#1\right)}}
\newcommand{\Cov}[1]{{\mathrm{Cov}[#1]}}

\TheoremsNumberedThrough     % Preferred (Theorem 1, Lemma 1, Theorem 2)
\ECRepeatTheorems

\EquationsNumberedThrough    % Default: (1), (2), ...
\MANUSCRIPTNO{}

\begin{document}
%%%%%%%%%%%%%%%%

\RUNTITLE{Non-Stationary Queues with Batch Arrivals}

% \TITLE{Bundling Information Goods of Decreasing Value}
\TITLE{
Non-Stationary Queues with Batch Arrivals
}

\ARTICLEAUTHORS{%
\AUTHOR{Andrew Daw, Brian Fralix, and Jamol Pender}
%, \URL{}}
\AFF{University of Southern California Marshall School of Business, Clemson University School of Mathematical and Statistical Sciences, and Cornell University School of Operations Research and Information Engineering}
%\AUTHOR{Authors blinded for review}
} % end of the block


\ABSTRACT{
Motivated by applications that involve setting proper staffing levels for multi-server queueing systems with batch arrivals, we present a thorough study of the queue-length process $\{Q(t); t \geq 0\}$, departure process $\{D(t); t \geq 0\}$, and the workload process $\{W(t); t \geq 0\}$ associated with the M$_{t}^{B_{t}}$/G$_{t}$/$\infty$ queueing system. With two fundamental assumptions of (non-stationary) Poisson arrivals and infinitely many servers, we otherwise maintain a highly general model, in which the service duration and batch size distributions may depend on time and, moreover, where the service durations within a batch may be arbitrarily dependent. Nevertheless, we find that the Poisson and infinite server assumptions are enough to show that for each $t > 0$, the law of $Q(t)$ is that of a weighted sum of mutually independent Poisson random variables. We further invoke this type of decomposition to derive various joint Laplace-Stieltjes transforms associated with the queue-length and departure processes. Next, we study the time-dependent behavior of the workload process, and we conclude by establishing almost sure convergence of the queue-length and workload processes (when properly scaled) to two different shot-noise processes, elevating the weak convergence results shown previously.
%, where arrivals occur in batches, with the batch size distribution varying with time.  Notably, we first show that both $Q(t)$ and $D(t)$ are equal in distribution to an infinite sum of independent, scaled Poisson random variables. When the batch size distribution has finite support, this sum becomes finite as well. We then derive the finite-dimensional distributions of both the queue-length process and the departure process, and we use these results to show that these finite-dimensional distributions converge weakly under a certain scaling regime, where the finite-dimensional distributions of the queue-length process converge weakly to a shot-noise process driven by a non-homogeneous Poisson process.  Next, we derive an expression for the joint Laplace-Stieltjes transform of $W(t)$, $Q(t)$, and $D(t)$, and we show that these three random variables, under the same scaling regime, converge weakly, where the limit associated with the workload process corresponds to another Poisson-driven shot-noise process.
}

\maketitle

%\author{
%\\Andrew Daw \\ Marshall School of Business, Data Sciences and Operations \\ University of Southern California
%\\ 401B Bridge Hall, Los Angeles, CA 90089 \\  andrew.daw@usc.edu
% \and
%  Brian Fralix \\ School of Mathematical and Statistical Sciences \\ Clemson University
%\\ O-110 Martin Hall, Box 340975, Clemson, SC 29634  \\  bfralix@clemson.edu
% \and
%  Jamol Pender \\ School of Operations Research and Information Engineering \\ Cornell University
%\\ 228 Rhodes Hall, Ithaca, NY 14853 \\  jjp274@cornell.edu
% }
%
%
%
%
%
%
%\noindent \textbf{Keywords:}  Non-Stationary Arrivals, Queueing Theory, Infinite Server Queues, Batch Arrivals, General Service, shot-noise, Simulation
%
%\noindent \textbf{2020 MSC:} 60K25, 60G55

%\tableofcontents

%**************************************************************************
%**************************************************************************

\section{Introduction}

The $M_{t}/G_{t}/\infty$ queueing system is arguably the most tractable time-varying queue studied in the literature, and it is described as follows.  Customers arrive to an area, consisting of infinitely many servers, in accordance to a non-homogeneous Poisson process with points $\{T_{n}\}_{n \geq 1}$ and arrival rate function $\lambda : [0, \infty) \rightarrow [0, \infty)$. Then, if a customer arrives to the system at time $t$, it brings with it a random amount of work having cumulative distribution function (CDF) $F_{t}$ for processing.  We assume that the $m^\text{th}$ arrival to the system occurs at time $T_{m}$, and it brings an amount of work $S_{m}$ for processing:  hence, conditional on $T_{m}$, the CDF of $S_{m}$ is $F_{T_{m}}$.  Finally, we let $\Lambda: [0, \infty) \rightarrow [0, \infty)$ denote the mean measure associated with the arrival process, where for each $t \geq 0$,
\begin{align} \Lambda(t) := \int_{0}^{t}\lambda(s)\mathrm{d}s.
\label{compDef}
\end{align}
For each real $t \geq 0$, let $Q(t)$ denote the number of customers present in the system at time $t$.  Generally $\{Q(t); t \geq 0\}$ is not a Markov process, yet it is well-known that when $Q(0) = 0$ (or when the law of $Q(0)$ is Poisson) the marginal distributions of $\{Q(t); t \geq 0\}$ are Poisson distributed:  more particularly, assuming $Q(0) = 0$ with probability one, it can be shown that for each $t > 0$,
\begin{eqnarray} \label{classicMGinfty}
 \mathbb{P}(Q(t) = k) = \frac{\left(\int_{0}^{t}\overline{F}_{s}(t-s)\lambda(s)\mathrm{d}s\right)^{k}e^{-\int_{0}^{t}\overline{F}_{s}(t-s)\lambda(s)\mathrm{d}s}}{k!}
  \end{eqnarray}
   for each integer $k \geq 0$, where $\overline{F}_{s}(u) := 1 - F_{s}(u)$ for each $u \geq 0$.

Formula (\ref{classicMGinfty}) can be proven in at least two different ways.  One approach involves making use of a time-dependent thinning property of non-homogeneous Poisson processes:  given a fixed $t > 0$, we say that if an arrival occurs at time $s \in (0,t]$, we `count' it with probability $p_{t}(s) := \overline{F}_{s}(t-s)$, independently of all other points in $(0,t]$.  Then $Q(t)$ is simply the number of counted points in $(0,t]$, which is Poisson distributed with mean
\begin{align*} \int_{0}^{t}\overline{F}_{s}(t-s)\lambda(s)\mathrm{d}s. \end{align*}
Another way to prove (\ref{classicMGinfty}) is to simply note that $\{(T_{n}, S_{n})\}_{n \geq 1}$ correspond to the points of a spatial Poisson process on $\mathbb{R}_{+}^{2}$, whose mean measure $\mu$ satisfies
\begin{eqnarray*}
 \mu((a,b] \times C) = \int_{a}^{b}\int_{C}\mathrm{d}F_{s}(u)\lambda(s)\mathrm{d}s
 \end{eqnarray*}
for each $a,b \in [0, \infty)$ satisfying $a < b$, and for each Borel measurable subset $C$ of $[0, \infty)$, where $\mathrm{d}F_{s}(u)$ denotes Lebesgue-Stieltjes integration with respect to the CDF $F_{s}$. {Once this has been established, $Q(t)$ corresponds to the number of points among $\{(T_{n}, S_{n})\}_{n \geq 1}$ found in the set $\{(x,y):  0 \leq x \leq t, x + y > t\}$.}  Our primary objective herein is to illustrate how the tractability of theses ideas translates to non-stationary, infinite-server queueing systems that receive arrivals in batches rather than in individual increments.

The research literature contains a large body of work addressing batch/bulk queueing systems.  To the best of our knowledge, the first study featuring queues with batch arrivals is that of \citet{miller1959contribution}.  Since then, many other papers have been written that feature a study of queues with batch arrivals that operate under various different conditions:  see for example  \citet{foster1964batched, shanbhag1966infinite, brown1969some, holman1983service, fakinos1984, chatterjeemukherjee1989, lucantoni1991new, takagi1991priority, economoufakinos1999, masuyama2002analysis, liu1993autocorrelations, lee1995batch, daw2018distributions}.   Later work has expanded the concept to a variety of related models, including priority queues and queues with server vacations.  There are other papers in the literature that establish heavy traffic limit theorems for queues with batch arrivals: examples include \citet{chiamsiri1981diffusion, pang2010two, pang2012infinite}.  These papers show that under certain conditions, one can approximate a properly-scaled queue length process with a diffusion process---such as Brownian motion and Ornstein-Uhlenbeck processes---and also show that these approximations can be applied to even multi-server and non-Markovian queues.  %However, our approach to studying the M$_{t}^{B_{t}}$/G$_{t}$/$\infty$ queue is exact and needs no approximations.

A recent application of batch queueing models is in the space of cloud-based data processing. In this case, the batches arriving to the system are collections of jobs submitted simultaneously. These jobs are then served by each being processed individually and returned. For more discussion, detailed models, and specific analysis for this setting, see works such as \citet{lu2011join, pender2016law, xie2017pandas, yekkehkhany2018gb} and references therein.  Another relevant application is in infectious disease modeling such as COVID-19, see for example \citet{kaplan2020covid, morozova2020model, palomo2020flattening}.  In this setting, the results for patients who potentially have COVID-19 arrive in a large batch to be processed at a facility.  Moreover, the data that we observe from COVID-19 is also of batch form as counts are made daily.  Finally, an emerging application of batch queues is in context of autonomous vehicles moving in platoons (batches) down highways and roads, e.g.~\citet{mirzaeian2018queueing, hampshire2020beyond}. Such applications also serve as the inspiration for the batch arrival queue staffing problem studied by \citet{daw2021staff}.

In this paper, we build upon ideas from three key prior works, \citet{economoufakinos1999}, \citet{de2017shot}, and \citet{daw2018distributions}, yielding distributional understanding of an essentially fully general infinite server queueing model with batch arrivals according to a non-stationary Poisson process. The earliest of the cornerstone concepts in this stream come from \citet{economoufakinos1999}, in which the authors find the probability generation function for the queue length and departure processes of the $M_t^{B}/G/\infty$ system. \citet{economoufakinos1999} recognized that this expression could hold for service times that are dependent within batches and independent across batches, possibly including different distributions of service for different types of customers within a batch. That paper itself builds on similar results from \citet{fakinos1984} for systems in which the service distributions are i.i.d.~(i.e., the $M_t^B/GI/\infty$ where $GI$ is meant to emphasize the independence assumption relative to $G$). However, no distributional equivalence was identified within these probability generation functions. By comparison, \citet{daw2018distributions} uncovered that this same idea can show that the random variable for the steady-state $M^B/GI/\infty$ queue length can be decomposed into sum of scaled Poisson random variables. However, the authors did not consider the departure or workload processes. In fact, of these three prior works, only \citet{de2017shot} consider the workload process. These authors study the $M_t^{B_t}/M/\infty$ system and connect both the queue length and the workload processes to shot-noise processes through the batch scaling limit, including on a process level through weak convergence of the finite-dimensional distributions. Their proof approach for the limit relies on Markov process theory through the convergence of generators. \citet{daw2018distributions} also considered the pointwise batch scaling limit, but only characterized the limiting generating function and did not offer any interpretation of the resulting stochastic process.

With these three prior works in mind, our goal here is to unite and extend these results to a highly general setting. That is, we study the $M_t^{B_t}/G_t/\infty$ system, allowing every distribution to potentially depend on time. Moreover, services can be arbitrarily dependent within batches and  may reflect different customer populations. We show that the distributional equivalence to a sum of scaled (independent) Poisson random variables (i.e., the idea from \citet{daw2018distributions}) holds in this fully general setting, both for the queue length process and the departure process.
%Moreover, this decomposition can also be used to study the workload (\textcolor{red}{I'M NOT SURE THIS SENTENCE IS TRUE, SHOULD WE REMOVE IT?}).
This leads to generalizations of the transform functions provided by \citet{economoufakinos1999}, {as our usage of a nonstationary thinning technique allows us to derive transforms that provide information on various finite-dimensional distributions associated with both $\{Q(t); t \geq 0\}$ and $\{D(t); t \geq 0\}$}.  {We then use a different technique to study the workload process of this infinite-server queue, and we conclude by showing} that the batch scaling limits first identified by \citet{de2017shot} can be elevated{, for both the queue-length process and the workload process,} to almost sure convergence.
%(\textcolor{red}{I REMOVED THE PHRASE WHERE WE MENTION THE GLIVENKO-CANTELLI THEROEM, BECAUSE WE DO NOT USE IT WHEN WE STUDY THE ALMOST-SURE LIMIT OF THE WORKLOAD}).
To the best of our knowledge, this is the first strong convergence result for a batch scaling limit.

Just like how many of the foremost benefits of the $M_t/G_t/\infty$ queue lie in the model's tractability for analysis and approximation of similar systems with limited capacity, we believe our results may be quite valuable to several interesting directions of related research. For example, consider the recent work on multi-server jobs, meaning queueing systems in which collections of jobs arrive together and also have a requirement that they must start together (see, e.g.,~\citet{rumyantsev2017stability,afanaseva2020stability,grosof2020stability,weng2020achieving,hong2021sharp,wang2021zero}, and references therein). The simultaneous start requirement is a salient model feature relative to batch arrival many server queues, but if there were infinitely many servers available then these models reduce to one another.  The multi-server jobs model is known to be quite challenging to analyze, so we offer our following analysis for any aid or insight into this problem through this more amenable, unlimited capacity setting. Naturally, much like the analog to many server queues, the infinite server queue provides an idealized bound for what is achievable. Similarly, batch arrival queues may also hold insight towards the design and management of system architectures for microservices, which have become increasingly common in cloud-based service (see, e.g.,~\citet{ueda2016workload,gan2018architectural,gan2021sage,lazarev2021dagger}). As a representative example, consider rideshare, where Uber describes their service as actually being comprised of over 2,200 microservices. (\citet{gluck2020introducing,chabbi2021crisp}). When a rider requests a ride, this  triggers a collection of related sub-tasks, such as those supporting matching, routing, billing, and user interfaces for both the passenger and the driver. In this case, the infinite server batch arrival queue provides a tractable model to analyze these microservices under ideal levels of support.

%**************************************************************************
%**************************************************************************

% \subsection{Main Contributions}
%
%In this paper, we make the following contributions to the queueing and applied probability literature:
%
%\begin{itemize}
%\item  We show a novel decomposition of the $M_t^{B_t}/G_t/\infty$ queue length distribution and the departure process in terms of sums of scaled Poisson random variables.
%\item We prove the finite dimensional distributions of the $M_t^{B_t}/G_t/\infty$ queue has a sum of scaled Poisson random variables representation.
%\item We prove an explicit expression for the auto-covariance of the queue length process and show that it is very different from the single arrival $M_t/G_t/\infty$ queue because of the batch dependence.
%\item We prove a batch scaling limit for the $M_t^{B_t}/G_t/\infty$ queue and show that the limit converges to a non-stationary shot-noise process.  The proof is direct and does not use convergence of generators.
%\end{itemize}

%**************************************************************************
%**************************************************************************

\subsection{Organization}

This paper is organized as follows. In Section~\ref{qDecompSec} we will first precisely define the queueing system and the three associated stochastic processes that serve as focal points of this work: the queue length, $Q(t)$, the departure process, $D(t)$, and the workload process, $W(t)$. Then, the remainder of this Section is devoted to establishing and exploring the decomposition of the queue into a weighted sum of mutually independent Poisson random variables. We provide these decompositions both at a single point in time and across a collection of epochs, yielding finite-dimensional perspectives. In Section~\ref{transformSec}, we use these decompositions to analyze  various transforms and performance metrics of the models, including the Laplace-Stieltjes transform (LST) and the auto-covariance, demonstrating the practicality of the Poisson random sum representation for proof methodology. We obtain these quantities across a variety of settings, including single and finitely many dimensions, and we provide examples of cases when the resulting expressions are quite simple. Finally, in Section~\ref{limitSec} we consider the batch scaling limits shown in the literature, and provide, to the best of our knowledge, the first {result addressing} almost sure convergence of batch arrival queues to shot-noise processes. In Section~\ref{concSec}, we discuss our results and conclude.
%In Section 2, we use the order statistics associated with the amounts of work found in each arriving batch to show that, for each $t \geq 0$, both $Q(t)$ and $D(t)$ can be expressed as an infinite sum of independent, scaled Poisson random variables.  We then continue by applying the same ideas to study the finite-dimensional distributions of both $\{Q(t); t \geq 0\}$ and $\{D(t); t \geq 0\}$, through the calculation of joint LST and auto-covariance functions, and we use these expressions to establish a weak-convergence result that builds on a scaling-limit result found in \citet{de2017shot}.  In Section 3, we analyze the joint LST of $W(t)$, $Q(t)$, and $D(t)$, and we use this expression to again derive a scaling limit result that builds on the scaling-limit result of \citet{de2017shot} in a different way.

\section{Defining and Decomposing the Queueing Model}\label{qDecompSec}

%\subsection{Model Definition: The $\mathbf{M_t^{B_t}/G_t/\infty}$ Queue}

In this paper, we will study three stochastic processes associated with the batch arrival infinite server queue: the queue length, departure, and workload processes. Throughout we consider the infinite-server queueing system $M_{t}^{B_{t}}/G_{t}/\infty$, where batches of customers arrive in accordance to a non-homogeneous Poisson process $\{A(t); t \geq 0\}$ with rate function $\lambda: [0, \infty) \rightarrow [0, \infty)$.  We denote the size (meaning number of customers) of the batch arriving at time $t$ as $B_{t}$, which is a random variable whose CDF may depend on $t$, and we assume that the amounts of work brought by customers within the batch has a joint distribution that may also depend on $t$.  In general, we allow amounts of work within a given batch to be arbitrarily dependent. No assumptions are placed on the distributions of work within a batch, but all batches are independent of each other.  We associate with this infinite-server system the stochastic processes $\{Q(t); t \geq 0\}$ and $\{D(t); t \geq 0\}$, where the queue length $Q(t)$ denotes the number of customers present in the system at time $t$, and where the departure process $D(t)$ denotes the number of service completions that occur over the interval $(0,t]$. Furthermore, we will let $\{W(t); t \geq 0\}$ be the workload process, where $W(t)$ denotes the total remaining service time of the customers present in the system at time $t$.

It will be of great use for us to carefully index and order the service durations by each customer and batch. For each integer $n \geq 1$, and each $j \in \{1,2,\ldots,n\}$, let $S_{j,n}(s)$ denote the amount of work brought by the $j^\text{th}$ customer contained in the batch of size $n$ that arrives at time $s$, and let $S_{j:n}(s)$ denote the $j^\text{th}$ smallest amount of work found in the same batch. Throughout the paper, we follow the convention that $S_{0:n}(s) = 0$ with probability one, and $S_{n+1:n}(s) = \infty$ with probability one.

\begin{figure}[h]
\centering
\includegraphics[width=.7\textwidth]{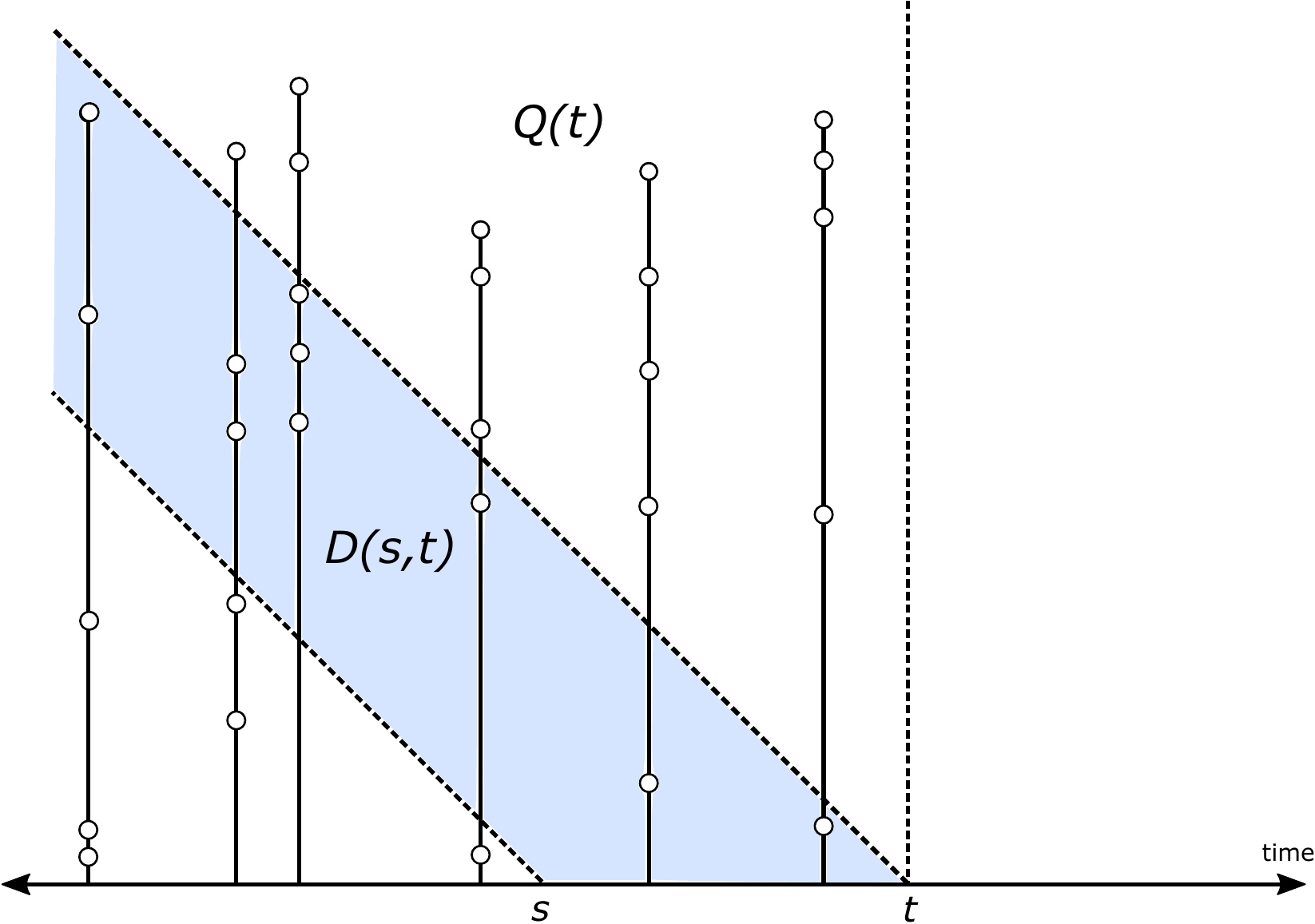}
\caption{A visualization of the $\mathbf{M_t^{B_t}/G_t/\infty}$  queueing system using Poisson random measures in the style of \citet{eick1993physics}, which motivates the thinning-based decomposition of the queue.}\label{thinFig}
\end{figure}

To motivate the sequel subsection's queueing decomposition visually, let us adapt the elegant Poisson random measure perspective shown in Figure 1 of \citet{eick1993physics}. In this diagram, the solid vertical lines mark the times of arrivals in the Poisson process. The dots along these lines then denote the lengths of the service durations within each arriving batch. Of course, by comparison to the $M_t/G/\infty$ queue considered by \citet{eick1993physics}, the batch arrivals mean that there are multiple service durations for each arrival epoch in the Poisson process. Because a customer is still in the system if her arrival time plus her service duration is greater than the current time, the total queue length is the number of points above the $45^\circ$ line. In this way, we can classify all of the arrival epochs up to $t$ by the number of jobs within a batch that remain in the system at time $t$, meaning the number of points above this line. This classification of the arrival times yields our thinning of the Poisson process.

\subsection{Decomposition into a Sum of Scaled Poissons}\label{thinSec}

%The thinning approach involves using the order statistics within each arriving batch to perform a thinning procedure to the non-homogeneous Poisson arrival process governing batch arrivals.  The specifics of the approach will depend on the random variable we wish to study, so for now we focus on explaining how the procedure can be used to study the law of $Q(t)$, when $Q(0) = 0$.

To formally define the decomposition of the queue, let us introduce a family of random variables sprouting from the definition of the customer-batch-indexed amounts of work random variables, $S_{j:n}(s)$. For each integer $j \in \{1,2,\ldots,n\}$, and each real $a,b$ satisfying $0 \leq a < b \leq t$, the random variable
\begin{align*}
Y_{j;n}(a,b] := \int_{(a,b]}\mathbf{1}(B_{s} = n, S_{j-1:n}(s) \leq t-s, S_{j:n}(s) > t-s)A(\mathrm{d}s) ,
\end{align*}
counts the number of batches arriving in the interval $(a,b]$ that are of size $n$, and are such that precisely $j-1$ customers within the batch have departed by time $t$, and precisely $n - (j-1)$ customers from this batch are still present in the system at time $t$.  We assign each such batch with the label $(j;n)$.  Using the fact that both the size of each batch and the amounts of work present in each batch are independent of all other batches, it follows from the classic thinning property of non-homogeneous Poisson processes that the random variables $\{Y_{j;n}(a,b]\}_{1 \leq j \leq n}$ are independent, Poisson random variables, where the mean of $Y_{j;n}(a,b]$ is given by
\begin{align}
\mathbb{E}[Y_{j;n}(a,b]] = \int_a^b \mathcal{P}_{s}(B_{s} = n, S_{j-1: n}(s) \leq t-s, S_{j:n}(s) > t-s)\lambda(s) \mathrm{d}s
,
\label{mainIntSingle}
\end{align}
with $\mathcal{P}_{s}$ as the probability measure associated with the batch of customers that arrive at time $s$.

Let us briefly comment on the integral in Equation~\eqref{mainIntSingle}. Here we have presented the mean of $Y_{j;n}(a,b]$ in a generality befitting of the weak assumptions we have made so far. This may be particularly true for the probability of the event inside this integral. For intuitions sake, let us give some example of how to compute this probability. Of course, if one has access to the conditional density of the order statistics, then this probability can be expressed accordingly. That is, letting $f_{n,s}(x_1, x_2, \dots, x_n)$ be the conditional density of the ordered service durations $S_{1:n}(s), S_{2:n}(s), \dots, S_{n:n}(s)$ given the batch arriving at time $s$ is of size $n$, then the probability of this event is given by the integral
\begin{align*}
\mathcal{P}_{s}\left(B_{s} = n, S_{j-1: n}(s) \leq \theta, S_{j:n}(s) > \theta\right)
&=
\int_0^{\theta} \dots \int_{x_{j-2}}^\theta \int_{\theta}^{\infty}  \dots \int_{x_{n-1}}^\infty
f_{n,s}(x_1, \dots, x_{j-1}, x_j, \dots, x_n)
\\&\qquad\cdot
\mathrm{d}x_1 \dots \mathrm{d}x_{j-1} \mathrm{d}x_j \dots \mathrm{d}x_n\,
\mathcal{P}_s(B_s = n)
.
\end{align*}
On the other hand, if $S_{j,n}(s)$ for $j \in \{1, 2, \dots, n\}$ are conditionally independent and identically distributed given the arrival epoch $s$ and resulting batch size $B_s = n$, then the probability of the event in Equation~\eqref{mainIntSingle} can be even more simply expressed. Letting $F_{n,s}(x) = \mathcal{P}_s\left(S_{1,n}(s) \leq x \mid B_s = n\right)$ and $\bar{F}_{n,s}(x) = 1 - F_{n,s}(x)$ for $x \geq 0$,  conditionally i.i.d. service durations imply that
$$
\mathcal{P}_{s}\left(B_{s} = n, S_{j-1: n}(s) \leq \theta, S_{j:n}(s) > \theta\right)
=
{n \choose j - 1} \left(F_{n,s}(\theta) \right)^{j-1} \left(\bar{F}_{n,s}(\theta)\right)^{n-j+1}
\mathcal{P}_s(B_s = n)
.
$$
This binomial coefficient form is known to arise for order statistics of independent and identically distributions, see e.g.~\citet{ross2014introduction}. Absent particular structural assumptions for a given domain or application, it remains a challenging research problem to express probabilities of order statistics more specifically while maintaining generality. Thus,  we will adhere to the flexible form used in Equation~\eqref{mainIntSingle}, but we will also refer to these useful specific examples for added context as needed. Broadly speaking, no further assumptions are needed to achieve our following decomposition, nor are further assumptions needed to use the decomposition in proof methodologies, as we will show. Nevertheless, throughout the remaining sections we will comment on how the resulting expressions can be simplified when assumptions are made.

%\begin{figure}[h]
%\centering
%\includegraphics[width=.7\textwidth]{physicsFigBase.eps}
%\caption{A visualization of the $M^B_t/G/\infty$  queueing system using Poisson random measures in the style of \citet{eick1993physics}, which allows us to represent the thinning perspective.}\label{thinFig}
%\end{figure}
%
%As a visualization of this decomposition of the queue, let us adapt the elegant Poisson random measure perspective shown in Figure 1 of \citet{eick1993physics}. In this diagram, the solid vertical lines mark the times of arrivals in the Poisson process. The dots along these lines then denote the lengths of the service durations within each arriving batch. Of course, by comparison to the $M_t/G/\infty$ queue considered by \citet{eick1993physics}, the batch arrivals mean that there are multiple service durations for each arrival epoch in the Poisson process. Because a customer is still in the system if her arrival time plus her service duration is greater than the current time, the total queue length is the number of points above the $45^\circ$ line. In this way, we can classify all of the arrival epochs up to $t$ by the number of jobs within a batch that remain in the system at time $t$, meaning the number of points above this line. This classification of the arrival times yields our thinning of the Poisson process.

Now, these terms and the style of reasoning in Figure~\ref{thinFig} lead us to our first result, the decomposition of $Q(t)$ into a sum of independent, scaled Poisson random variables. Furthermore, by immediate consequence of the queue's decomposition, we can find an analogous representation for the departure process.

\begin{theorem} \label{SACTQandD}
For each $t > 0$,
\begin{eqnarray*}
Q(t) = \sum_{n=1}^{\infty}\sum_{j=1}^{n+1}(n - j + 1)Y_{j;n}(0,t], ~~~~~~~ D(t) = \sum_{n=1}^{\infty}\sum_{j=1}^{n+1}(j - 1)Y_{j;n}(0,t]
,
\end{eqnarray*}
where the random variables $\{Y_{j;n}(0,t] \mid j, n \in \mathbb{Z}_+, j \leq n\}$ are mutually independent.
\end{theorem}
\proof{Proof.} If a batch arriving in the interval $(0,t]$ is assigned label $(j;n)$, precisely $(n - j + 1)$ customers within that batch are still present at time $t$, and precisely $j-1$ of those customers have departed by time $t$.  Hence, the number of customers present in the system at time $t$ from a batch with label $(j;n)$ is $(n - j + 1)Y_{j;n}(0,t]$, and the number of departures over $(0,t]$ of customers from a batch with label $(j;n)$ is $(j-1)Y_{j;n}(0,t]$. Summing over all possible labels completes the proof.
\hfill\Halmos \endproof

\noindent \textbf{Remark} Readers may observe that we include the term $(0)Y_{n+1;n}(0,t]$, as well as the term $(0)Y_{1;n}(0,t]$ in $Q(t)$ and $D(t)$, respectively, which seems unnecessary since both terms are clearly zero with probability one.  However, following this convention will make it easier later to express various joint Laplace-Stieltjes transforms associated with both $\{Q(t); t \geq 0\}$ and $\{D(t); t \geq 0\}$.

\subsection{Contrasting with the Queue Length under Individual Arrivals}

Now that we have shown the Poisson decomposition of the queue length and departure processes, we can use the representation to analyze the covariance between the two processes.

\begin{proposition} \label{CovQandD}
For each $t > 0$, the covariance of $Q(t)$ and $D(t)$ is given by
\begin{eqnarray*}
\mathrm{Cov}(Q(t), D(t)) = \sum_{n=1}^{\infty}\sum_{j=1}^{n+1}(n-j+1)(j-1)\int_{0}^{t}\mathcal{P}_{s}(B_{s} = n, S_{j-1:n}(s) \leq t-s, S_{j:n}(s) > t-s)\lambda(s)\mathrm{d}s.
 \end{eqnarray*}
\end{proposition}
\proof{Proof.}
The proof of this result exploits properties of Poisson processes and the decomposition of the queue length and departure process given in Theorem~\ref{SACTQandD}. By the sum of Poisson's representation and the definition of covariance we have that
\begin{align*}
\mathrm{Cov}(Q(t), D(t)) &= \mathrm{Cov} \left( \sum_{n=1}^{\infty}\sum_{j=1}^{n+1}(n - j + 1)Y_{j;n}(0,t], \sum_{n=1}^{\infty}\sum_{j=1}^{n+1}(j - 1)Y_{j;n}(0,t]\right) \\
&= \sum_{n=1}^{\infty}\sum_{j=1}^{n+1}(n - j + 1) (j - 1) \mathrm{Cov} \left( Y_{j;n}(0,t], Y_{j;n}(0,t] \right)
.
\end{align*}
Because the $Y$'s are independent, the covariance between separate variables is 0 and hence this reduces to simply summing over the variances. Since the variance of a Poisson is simply its mean, we find
\begin{align*}
\mathrm{Cov}(Q(t), D(t))
&= \sum_{n=1}^{\infty}\sum_{j=1}^{n+1}(n - j + 1) (j - 1) \mathrm{Var} \left( Y_{j;n}(0,t]\right) \\
&= \sum_{n=1}^{\infty}\sum_{j=1}^{n+1}(n - j + 1) (j - 1) \mathrm{E} \left[ Y_{j;n}(0,t] \right] \\
 &= \sum_{n=1}^{\infty}\sum_{j=1}^{n+1}(n-j+1)(j-1)\int_{0}^{t}\mathcal{P}_{s}(S_{j-1:n}(s) \leq t-s, S_{j:n}(s) > t-s)\mathcal{P}_{s}(B_{s} = n)\lambda(s)\mathrm{d}s
 ,
\end{align*}
completing the proof.
\hfill\Halmos \endproof

  \begin{figure}[h]
\centering
\includegraphics[width=.7\textwidth]{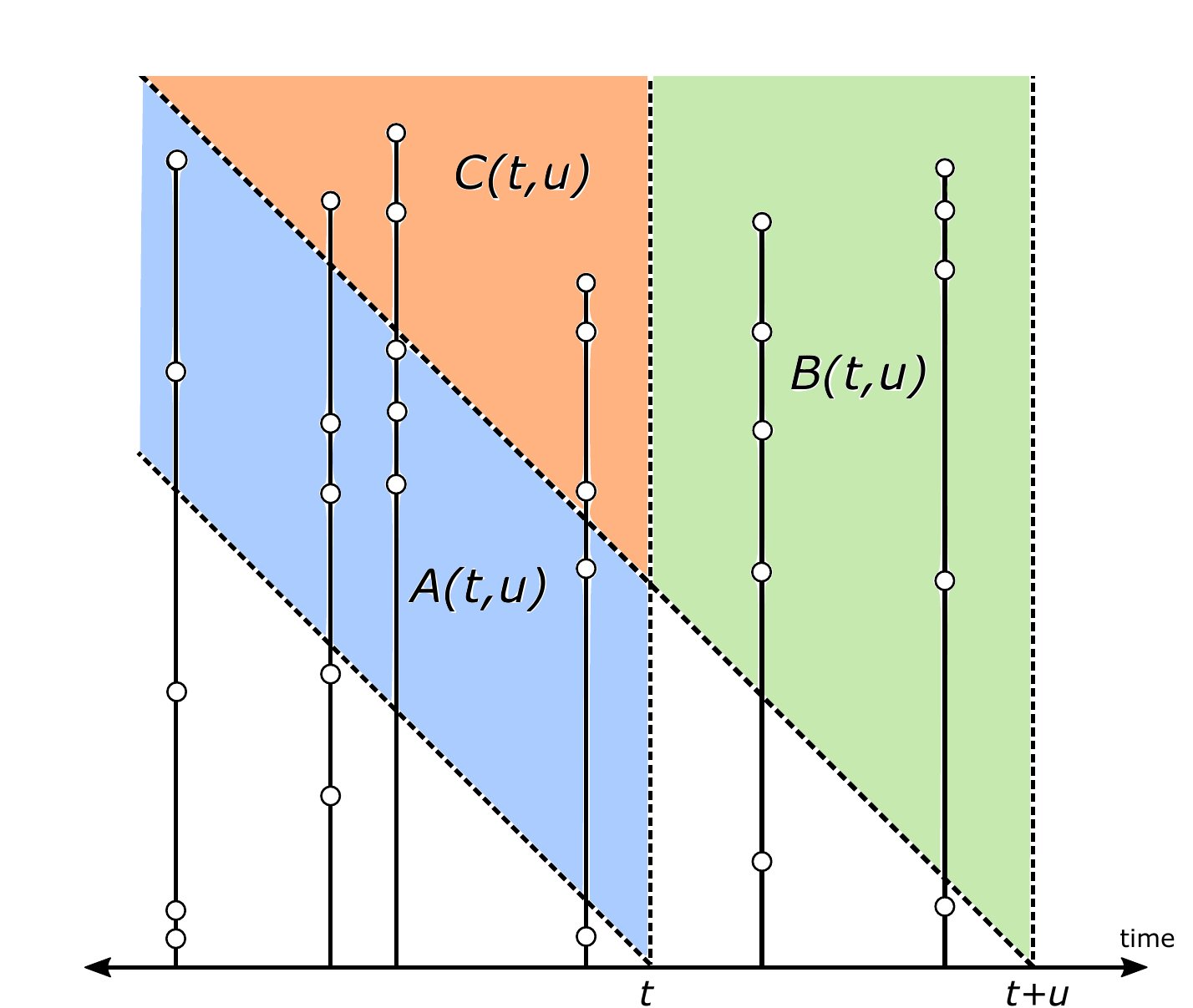}
\caption{A Poisson random measure perspective on the overlap of customers in a $\mathbf{M_t^{B_t}/G_t/\infty}$  queueing system inspected at times $\boldsymbol{t}$ and $\boldsymbol{t+u}$.}\label{figABC}
\end{figure}

Proposition \ref{CovQandD} shows that in general, $Q(t)$ and $D(t)$ are positively correlated, but when all batches are of size one with probability one, the result shows that
\begin{eqnarray*}
\mathrm{Cov}(Q(t), D(t)) = 0.
 \end{eqnarray*} which is a well-known result that is addressed in e.g. \citet{eick1993physics}. The same Poisson random measure visualizations in Figure~\ref{thinFig} can reveal this dependence as well and demonstrate the difference between the batch and individual arrival settings, as shown in Figure~\ref{figABC}. Let time $t$ and offset $u$ be fixed. Following the decomposition used in \citet{eick1993physics}, let us introduce the quantities $A(t,u)$, $B(t,u)$, and $C(t,u)$ defined such that $A(t,u)$ is the number of entities that arrive by time $t$ and depart in the interval $[t, t+u]$, $B(t,u)$ is the number of entities arriving in $[t, t+u]$ that remain in the system at time $t + u$, and $C(t,u)$ is the number of entities arriving by time $t$ that remain in the system at time $t+u$. Then, by definition we have that  $Q(t) = A(t,u) + C(t,u)$ and $Q(t+u) = B(t,u) + C(t,u)$. By the independent increments of the Poisson process, we can note that $B(t,u)$ is independent from $A(t,u)$ and $C(t,u)$. However, unlike the $M_t/G/\infty$ model studied in \citet{eick1993physics}, $A(t,u)$ is not independent from $C(t,u)$. This is a consequence of the batch arrivals, as there is dependency between the ordered service times within one batch. Using these definitions, we can express  the auto-covariance  in terms of these regions as
 \begin{align*}
 \Cov{Q(t), Q(t+u)}  &= \E{(A(t,u) + C(t,u))(B(t,u) + C(t,u))}
- \E{A(t,u) + C(t,u)}\E{B(t,u) + C(t,u)}
 \\
 &= \Cov{A(t,u),C(t,u)}  + \Var{C(t,u)}
 .
 \end{align*}

\subsection{Recognizing the Decomposition in Finite-Dimensional Perspectives}

This thinning technique can also be used to derive the joint finite-dimensional distributions of both $\{Q(t); t \geq 0\}$ and $\{D(t); t \geq 0\}$, but doing so requires a more elaborate thinning procedure.  Given a collection of real numbers $\{t_{\ell}\}_{\ell = 1}^{m}$ satisfying $0 < t_{1} < t_{2} < \ldots < t_{m}$ and an integer $n \geq 1$, we define the random variable $Y_{j_{k}, j_{k+1}, \ldots, j_{m};n}$ as
\begin{eqnarray}
Y_{j_{k}, j_{k+1}, \ldots, j_{m};n}(a,b] := \int_{(a,b]}\mathbf{1}\left(\bigcap_{\ell = k}^{m}\{S_{j_{\ell} - 1:n}(s) \leq t_{\ell} - s, S_{j_{\ell}:n}(s) > t_{\ell} - s\}\right)A(\mathrm{d}s)
\label{fdYdef}
.
\end{eqnarray}
Here, $Y_{j_{k}, j_{k+1}, \ldots, j_{m};n}(a,b]$ should be interpreted as the number of batches of size $n$ that arrive in the interval $(a,b]$ satisfying the property that for each $\ell \in \{k,k+1,\ldots,m\}$, exactly $j_{\ell} - 1$ customers from the batch have departed from the system at time $t_{\ell}$ (meaning also that exactly $n - j_{\ell} + 1$ customers from the batch are still present in the system at time $t_{\ell}$).  Using well-known thinning properties of non-homogeneous Poisson processes, we can say that $Y_{j_{k}, j_{k+1}, \ldots, j_{m}}(a,b]$ is a Poisson random variable that satisfies
\begin{align}
\mathbb{E}[Y_{j_{k}, j_{k+1}, \ldots, j_{m};n}(a,b]] = \int_{a}^{b}\mathcal{P}_{s}\left(\{B_{s} = n\} \cap \bigcap_{\ell = k}^{m}\{S_{j_{\ell} - 1:n}(s) \leq t_{\ell} - s, S_{j_{\ell}:n}(s) > t_{\ell} - s\}\right)\lambda(s)\mathrm{d}s.
\label{mainIntMulti}
\end{align}
These random variables contribute value to each $Q(t_{\ell})$, as well as to each $D(t_{\ell})$ value, for $1 \leq \ell \leq m$. Theorem~\ref{SACTQandD-FD} provides the finite-dimensional analog to the single dimensional decomposition in Theorem~\ref{SACTQandD}, and the proof follows similarly.

\begin{theorem} \label{SACTQandD-FD}
Given $m \in \mathbb{Z}_+$, let $\mathcal{J}_{\ell,m,n} = \{ j_\ell, \ldots, j_m \in \{1, \dots, n+1\} \mid j_{i-1} \leq j_i \leq n + 1 ~ \forall ~ \ell+1 \leq i \leq m\}$ for each $n \in \mathbb{Z}_+$. Then, for every collection of $m$ epochs $0 < t_1 < t_2 < \ldots < t_m$,
\begin{align*}
%Q(t) = \sum_{n=1}^{\infty}\sum_{j=1}^{n+1}(n - j + 1)Y_{j;n}(0,t]
Q(t_i)
&=
\sum_{n=1}^\infty
\sum_{\ell=1}^i
\sum_{j_\ell, \dots, j_m \in \mathcal{J}_{\ell,m,n}}
(n - j_i + 1) Y_{j_\ell, j_{\ell+1}, \dots, j_m; n}(t_{\ell-1}, t_\ell]
&&
\forall ~ i \in \{1, \dots, m\}
,
\end{align*}
and
\begin{align*}
D(t_i)
&=
\sum_{n=1}^\infty
\sum_{\ell=1}^i
\sum_{j_\ell, \dots, j_m \in \mathcal{J}_{\ell,m,n}}
(j_i - 1) Y_{j_\ell, j_{\ell+1}, \dots, j_m; n}(t_{\ell-1}, t_\ell]
&&
\forall ~ i \in \{1, \dots, m\}
,
\end{align*}
where all the Poisson random variables are mutually independent.
\end{theorem}

As a consequence of this finite-dimensional decomposition, we can see that the Poisson random variables defined in Equation~\eqref{fdYdef} elucidate the dependence between observations of the process across epochs. That is, because the $Y$'s are mutually independent random variables, the precise dependence of the queue (and departure process) at different moments in time is distilled to any shared copies of particular $Y$ variables that appear in each epoch's decomposition. We will explore this concept further through auto-covariances and finite dimensional transforms that we provide in the next section.

As we did for the single dimensional case, let us briefly comment on the integral within Equation~\eqref{mainIntMulti}. As before, if the conditional joint density of the ordered service durations is in hand, then we can express the probability of this event accordingly. That is,
\begin{align*}
&
\mathcal{P}_{s}\left(\{B_{s} = n\} \cap \bigcap_{\ell = k}^{m}\{S_{j_{\ell} - 1:n}(s) \leq \theta_\ell, S_{j_{\ell}:n}(s) > \theta_\ell\}\right)
\\
&
\quad=
\int_0^{\theta_k} \dots \int_{x_{j_k-2}}^{\theta_k} \int_{\theta_k}^{\theta_{k+1}} \int_{x_{j_k}}^{\theta_{k+1}}  \dots \int_{x_{j_m-2}}^{\theta_m} \int_{\theta_m}^{\infty} \dots \int_{x_{n-1}}^\infty
f_{n,s}(x_1, \dots, x_{j_k-1}, x_{j_k}, x_{j_k+1}, \dots, x_{j_m-1}, x_{j_m}, \dots, x_n)
\\&\qquad\cdot
\mathrm{d}x_1 \dots \mathrm{d}x_{j_k-1} \mathrm{d}x_{j_k} \mathrm{d}x_{j_k+1} \dots \mathrm{d}x_{j_m-1} \mathrm{d}x_{j_m} \dots \mathrm{d}x_n \,
\mathcal{P}_s(B_s = n)
.
\end{align*}
Furthermore, if the service durations are conditionally independent and identically distributed given then arrival epoch $s$ and batch size $B_s = n$, this probability can again be simply expressed. Here, we have that
\begin{align*}
&
\mathcal{P}_{s}\left(\{B_{s} = n\} \cap \bigcap_{\ell = k}^{m}\{S_{j_{\ell} - 1:n}(s) \leq \theta_\ell, S_{j_{\ell}:n}(s) > \theta_\ell\}\right)
\\
&\quad=
{n \choose {j_k - 1, j_{k+1} - j_k, \dots, j_m - j_{m-1}, n - j_m + 1}}
\left(F_{n,s}(\theta_k)\right)^{j_k-1}
\left(F_{n,s}(\theta_{k+1}) - F_{n,s}(\theta_{k})\right)^{j_{k+1}-j_k}
\\
&\qquad\cdot
\dots
\left(F_{n,s}(\theta_{m}) - F_{n,s}(\theta_{m-1})\right)^{j_{k+1}-j_k}
\left(\bar{F}_{n,s}(\theta_m)\right)^{n-j_m+1}
.
\end{align*}

\section{Transforms of the Stochastic Processes}\label{transformSec}

In addition to immediate benefits for simulation of the queueing model, the decompositions in Theorems~\ref{SACTQandD} and~\ref{SACTQandD-FD} are also quite valuable methodologically for analysis of the underlying stochastic processes. We will demonstrate this throughout the remainder of the paper. In this section, we will show this through analysis of transforms and various performance metrics of the models. Both for comparison's sake and for use downstream, we will obtain LST expressions for the process in two slightly different settings and through two different approaches: a classical conditional uniformity argument and the decomposition technique we have established in Theorem~\ref{SACTQandD-FD}. We begin with the conditional uniformity approach.
%In this section, we begin by demonstrating how these sum of scaled Poisson's representations can be used to obtain transforms and related performance metrics of the models.  %\textbf{Are we using the decomposition in this section?  Maybe it is better to say something like the following:  ``While the sum of scaled Poisson representation does not seem to carry over in a natural way to the workload process, standard conditional uniformity arguments can be used to study the distribution of the random vector $(W(t), Q(t), D(t))$ for each $t \geq 0$.''}

%\tr{add commentary about first using standard conditional uniformity and then decomposition}

\subsection{Joint Transform of the Workload, Queue, and Departures}

%\tr{comment on using conditional uniformity arguments that are more standard, but the understanding of the workload will be useful later}

We now present an alternative approach towards studying, for the $M_{t}^{B_{t}}/G_{t}/\infty$ queueing system, the joint LST of $W(t)$, $Q(t)$, and $D(t)$. By comparison to the sum-of-Poissons decomposition, here we keep track of {\em all} arrivals in $(0,t]$, then use indicator functions to describe $W(t)$, $Q(t)$, and $D(t)$ once we know when all arrivals occur in $(0,t]$.  Doing the calculations in this way allows us to more easily work with random batch sizes, and moreover this conditional uniformity approach allows us to analyze the workload process as well.

%In this section, we compute the joint Laplace-Stieltjes transforms of the queue length, workload process, and the departure process.  Our method is based on our novel Poisson process representation of the queue queue length, workload, and departure processes.  We find that our perspective of the joint Laplace-Stieltjes transform yields insights about the dependence between all of the processes.  Our first result in this regard is given below.

\begin{theorem} \label{JointLSTofQWD} For each $\alpha \geq 0$, each $\beta \geq 0$, and each $\gamma \geq 0$,
\begin{align} \label{JointLSTofQWDEq} \mathbb{E}[e^{-\alpha W(t) - \beta Q(t) - \gamma D(t)}] = \exp\left(-\int_{0}^{t}\mathcal{E}_{s}\left[1 - e^{-\sum_{j=1}^{B_{s}}[\beta\mathbf{1}(S_{j}(s) > t-s) + \gamma\mathbf{1}(S_{j}(s) \leq t-s) + \alpha (S_{j}(s) - (t-s))^{+}]}\right]\lambda(s)\mathrm{d}s\right) \end{align} where $\mathcal{E}_{s}$ denotes expectation, conditional on having a batch arrival at time $s$.
\end{theorem}

\proof{Proof.} Conditioning on $A(t)$ yields
\begin{align} \label{JointLSTofQWDStep1} \mathbb{E}[e^{-\alpha W(t) - \beta Q(t) - \gamma D(t)}] = e^{-\Lambda(t)} + \sum_{m=1}^{\infty}\mathbb{E}[e^{-\alpha W(t) - \beta Q(t) - \gamma D(t)} \mid A(t) = m]\frac{\Lambda(t)^{m}e^{-\Lambda(t)}}{m!}, \end{align} where the compensator $\Lambda(t)$ is as defined in Equation~\eqref{compDef}. Next, recall that conditional on $A(t) = m$, the arrival times $T_{1}, T_{2}, \ldots, T_{m}$ are equal in distribution to the order statistics associated with $m$ i.i.d.~absolutely continuous random variables:  the conditional joint PDF of $T_{1}, T_{2}, \ldots, T_{m}$, given $A(t) = m$, is known to be
\begin{align*} f_{T_{1}, T_{2}, \ldots, T_{m} \mid A(t) = m}(s_{1}, s_{2}, \ldots, s_{m}) = \left\{
                                                                 \begin{array}{ll}
                                                                   m!\prod_{\ell = 1}^{m}\frac{\lambda(s_{\ell})}{\Lambda(t)}, & \hbox{$0 < s_{1} < s_{2} < \ldots < s_{m} < t$;} \\
                                                                   0, & \hbox{otherwise.}
                                                                 \end{array}
                                                               \right. \end{align*}
        Then for each $m \geq 1$,
\begin{align} \label{JointLSTofQWDStep2} &~~ \mathbb{E}[e^{-\alpha W(t) - \beta Q(t) - \gamma D(t)} \mid A(t) = m]
\\ &= \int_{0}^{t}\int_{s_{1}}^{t}\cdots \int_{s_{m-1}}^{t}\mathcal{E}_{s_{1}, \ldots, s_{m}}[e^{-\sum_{\ell = 1}^{m}\sum_{j=1}^{B_{s_{\ell}}}[\alpha (S_{j}(s_{\ell}) - (t - s_{\ell}))^{+} + \beta\mathbf{1}(S_{j}(s_{\ell}) > t- s_{\ell}) + \gamma\mathbf{1}(S_{j}(s_{\ell}) \leq t - s_{\ell})]}]m!\prod_{\ell = 1}^{m}\frac{\lambda(s_{\ell})}{\Lambda(t)}\mathrm{d}s_{m}\ldots \mathrm{d}s_{2}\mathrm{d}s_{1} \nonumber \end{align} where $\mathcal{E}_{s_{1}, \ldots, s_{m}}$ represents conditional expectation, given batches arrive at times $s_{1}, s_{2}, \ldots, s_{m}$.  Furthermore, since batches are independent,
\begin{eqnarray} \label{JointLSTofQWDStep3}
\lefteqn{\mathcal{E}_{s_{1}, \ldots, s_{m}}[e^{-\sum_{\ell = 1}^{m}\sum_{j=1}^{B_{s_{\ell}}}[\alpha (S_{j}(s_{\ell}) - (t - s_{\ell}))^{+} + \beta\mathbf{1}(S_{j}(s_{\ell}) > t- s_{\ell}) + \gamma\mathbf{1}(S_{j}(s_{\ell}) \leq t - s_{\ell})]}] }  && \nonumber \\
&=& \prod_{\ell = 1}^{m}\mathcal{E}_{s_{\ell}}[e^{-[\alpha \sum_{j=1}^{B_{s_{\ell}}}(S_{j}(s_{\ell}) - (t - s_{\ell}))^{+} + \beta\mathbf{1}(S_{j}(s_{\ell}) > t - s_{\ell}) + \gamma\mathbf{1}(S_{j}(s_{\ell}) \leq t - s_{\ell})]}]
,
\end{eqnarray}
and this proves that the integrand of the multiple integral found in (\ref{JointLSTofQWDStep2}) is a symmetric function on $[0,t]^{m}$.  Hence, (\ref{JointLSTofQWDStep3}) simplifies to
\begin{align} \label{JointLSTofQWDStep4} \mathbb{E}[e^{-\alpha W(t) - \beta Q(t) - \gamma D(t)} \mid A(t) = m] = \frac{1}{\Lambda(t)^{m}}\left(\int_{0}^{t}\mathcal{E}_{s}[e^{-\sum_{j=1}^{B_{s}}[\alpha (S_{j}(s) - (t - s))^{+} + \beta\mathbf{1}(S_{j}(s) > t - s) + \gamma\mathbf{1}(S_{j}(s) \leq t-s)]}]\lambda(s)\mathrm{d}s\right)^{m} \end{align} and after plugging (\ref{JointLSTofQWDStep4}) into (\ref{JointLSTofQWDStep1}) and simplifying, we get
\begin{align*} \mathbb{E}[e^{-\alpha W(t) - \beta Q(t) - \gamma D(t)}] = \exp\left(-\int_{0}^{t}\mathcal{E}_{s}\left[1 - e^{-\sum_{j=1}^{B_{s}}[\beta\mathbf{1}(S_{j}(s) > t-s) + \gamma\mathbf{1}(S_{j}(s) \leq t-s) + \alpha (S_{j}(s) - (t-s))^{+}]}\right]\lambda(s)\mathrm{d}s\right) \end{align*} proving Theorem \ref{JointLSTofQWD}. \hfill\Halmos \endproof

In light of Theorem \ref{JointLSTofQWD}, it is not difficult to see that the joint LST of $W(t)$, $Q(t)$, and $D(t)$ simplifies significantly under the additional assumption that within a batch arriving at time $s$, the amounts of work are all i.i.d. with CDF $F_{s}$.

\begin{corollary} \label{Cor1QWD}  Suppose that when a batch arrives at time $s$, each customer within that batch brings a generally distributed amount of work with CDF $F_{s}$, independently of everyone else.  Next, for each $s,t \geq 0$, let $X_{s}$ be a random variable whose CDF is $F_{s}$, and define the Laplace-Stieltjes transform
\begin{align*} \phi_{s,t}(\alpha) := \mathbb{E}[e^{-\alpha (X_{s} - t)} \mid X_{s} > t] \end{align*}  Then for $\alpha \geq 0$, $\beta \geq 0$, and $\gamma \geq 0$,
\begin{align} \label{Cor1QWDEq} \mathbb{E}[e^{-\alpha W(t) - \beta Q(t) - \gamma D(t)}] = \exp\left(-\int_{0}^{t}\mathcal{E}_{s}\left[1 - \left[F_{s}(t-s)e^{-\gamma} + \phi_{s,t-s}(\alpha)e^{-\beta}\overline{F}_{s}(t-s)\right]^{B_{s}}\right]\lambda(s)\mathrm{d}s\right). \end{align}
\end{corollary}  Even though each LST $\phi_{s,t}(\alpha)$ typically does not simplify much further, it is noteworthy to realize that $\phi_{s,t}(\alpha)$ can be expressed reasonably well for the special case where $F_{s}$ is the CDF of a phase-type random variable.  In particular, when $F_{s}$ is the CDF of an exponentially distributed random variable with rate $\mu_{s}$, we get
\begin{align*} \phi_{s,t}(\alpha) = \frac{\mu_{s}}{\mu_{s} + \alpha}. \end{align*}  In the next corollary, we use this simple fact to show that the joint LST of $Q(t)$ and $W(t)$ simplifies considerably when all amounts of work are exponentially distributed.

\begin{corollary} \label{Cor2QWD} Suppose that when a batch arrives at time $s$, each customer within that batch brings an independently and exponentially distributed amount of work with rate $\mu_{s}$.  Then for $\alpha \geq 0$, $\beta \geq 0$, and $\gamma \geq 0$,
\begin{align*} \mathbb{E}[e^{-\alpha W(t) - \beta Q(t) - \gamma D(t)}] = \exp\left(-\int_{0}^{t}\mathcal{E}_{s}\left[1 - \left[(1 - e^{-\mu_{s}(t-s)})e^{-\gamma} + \frac{\mu_{s}}{\mu_{s} + \alpha}e^{-\beta}e^{-\mu_{s}(t-s)}\right]^{B_{s}}\right]\lambda(s)\mathrm{d}s\right). \end{align*}  \end{corollary}

\subsection{Connecting to the Workload under Individual Arrivals}

%\textbf{Should we place this section after we analyze the joint LST of the random vector $(W(t), Q(t), D(t))$ (while changing the first sentence of course)?}

As a brief aside, let us take a moment to contrast the transforms we have just found with prior results on the workload process of time-varying infinite-server queueing systems. The literature on the infinite server workload process under batch arrivals appears to be scarce; the same seems true even when each batch is of size one (individual arrivals).  The most relevant reference we found that even remotely addresses the workload process of time-varying infinite-server queues with Poisson arrivals is \citet{goldberg2008last}, which is primarily concerned with the study of the last departure time from a $M_{t}/G/\infty$ queueing system, when the arrival process stops at some fixed, deterministic time $\tau$.  While \citet{goldberg2008last} do not study the workload process in itself, Theorem 2.1 of \citet{goldberg2008last} can be used to derive the LST of the workload process of the $M_{t}/G/\infty$ queue, as this result provides the conditional joint distribution, given $Q(t) = n$, of the remaining service times of the $n$ customers present in the system at time $t$. The next proposition is a slight generalization of Theorem 2.1 of \citet{goldberg2008last}, in that it applies to the $M_{t}/G_{t}/\infty$ system, and it can be proven in precisely the same manner, which involves conditioning on the order statistics associated with the thinned Poisson process associated with customers that are still present in the system at time $t$, then simplifying.  We omit the details.

\begin{proposition} \label{residualservices} Conditional on $Q(t) = n$, the remaining service times of the customers present at time $t$ are iid, with CDF $H_{t}: \mathbb{R} \rightarrow [0,1]$ having tail
\begin{align*} \overline{H}_{t}(x) = \frac{1}{\nu_{t}}\int_{0}^{t}\overline{F}_{s}(t + x - s)\lambda(s)\mathrm{d}s \end{align*} where
\begin{align*} \nu_{t} := \int_{0}^{t}\overline{F}_{s}(t - s)\lambda(s)\mathrm{d}s. \end{align*} \end{proposition}  Once Proposition \ref{residualservices} is known, it can be used to calculate the LST of $W(t)$ for the $M_{t}/G_{t}/\infty$ queue.  Again, we omit the proof as it follows from conditioning on $Q(t)$, then applying Proposition \ref{residualservices}.

\begin{proposition} The LST of $W(t)$ is as follows: for each $\alpha \geq 0$,
\begin{align*} \mathbb{E}[e^{-\alpha W(t)}] = e^{-(1 - \phi_{t}(\alpha))\int_{0}^{t}\lambda(s)\overline{F}_{s}(t-s)\mathrm{d}s} \end{align*} where $\phi_{t}$ is the LST associated with the CDF $H_{t}$. \end{proposition}

As an interesting aside, it is worth noting that Theorem 2.2 of \citet{goldberg2008last} (namely, Identity (2.5)) can be derived with our thinning approach from Section~\ref{thinSec}, without applying Proposition \ref{residualservices}.

\begin{proposition} \label{GoldbergWhittDeparture} Let $D$ denote the last departure time of a $M_{t}/G_{t}/\infty$ queue when arrivals are turned off at time $t$, and let $T_{t} := (D - t)^{+}$ denote the remaining amount of time after $t$ until the last departure.  Then for each $x \geq 0$,
\begin{align*} \mathbb{P}(T_{t} \leq x) = e^{-\nu_{t}\overline{H}_{t}(x)}. \end{align*} \end{proposition}
\proof{Proof.} Fix $x \geq 0$, and let $Y^{(t+x)}(0,t]$
%(\textbf{Note the change from $Y_{0}$ to $Y^{(t+x)}$ in the next few lines.  In the previous draft, we defined $Y_{0}$, then used $Y_{1}$ a few times})
denote the number of jobs that arrive in the interval $(0,t]$ that are still present in the system at time $t + x$.  This random variable is a Poisson random variable with mean
\begin{align*} \mathbb{E}[Y^{(t+x)}(0,t]] = \int_{0}^{t}\overline{F}_{s}(t + x - s)\lambda(s)\mathrm{d}s \end{align*} which implies
\begin{align*} \mathbb{P}(T_{t} \leq x) = \mathbb{P}(Y^{(t+x)}(0,t] = 0) = e^{-\int_{0}^{t}\overline{F}_{s}(t + x - s)\lambda(s)\mathrm{d}s} = e^{-\nu_{t}\overline{H}_{t}(x)} \end{align*} proving the claim. \hfill\Halmos \endproof

Finer understanding of the infinite server workload process will be of use when we return to studying the process in the batch scaling limits analyzed in Section~\ref{limitSec}.

\subsection{Finite-Dimensional Joint Transform and Auto-Covariances}

%\textbf{Should we place these results here?  It seems natural to include this at the end of Section 2, after we discuss the more elaborate thinning procedure.  I remember the referees criticizing us for studying the joint LST of $Q(t)$, $D(t)$, and $W(t)$ using conditional uniformity instead of the scaled Poisson decomposition, but maybe we could start Section 3 by mentioning we will eventually relate scaling limits of the queue length process to the workload process.}

%\tr{comment on now switching to decomp-based methods}

%Just as we moved the decomposition from a single dimension in Theorem~\ref{SACTQandD} to multiple in Theorem~\ref{SACTQandD-FD}, we can now do the same for the transforms that leverage these decompositions. For simplicity, we will focus on the queue length and departure process here.

In the preceding techniques, conditional uniformity allowed us to access all batch arrival times, which made the workload process straightforward to analyze. By comparison, the decomposition in Theorems~\ref{SACTQandD} and~\ref{SACTQandD-FD} does not provide such access. However, Theorem~\ref{SACTQandD-FD} does provide simple access to the laws of the queue and the departure processes at a collection of points in time --- rather than at only single epoch as in Theorem~\ref{JointLSTofQWD} --- and this enables us to easily study the stochastic processes from finite-dimensional perspectives. Our next result uses the decomposition to provide an expression for the joint LST of the finite-dimensional distributions of both $\{Q(t); t\geq 0\}$ and $\{D(t); t \geq 0\}$, as well as the auto-covariance functions of both $\{Q(t); t \geq 0\}$ and $\{D(t); t \geq 0\}$.

\begin{theorem} \label{MainLSTresult}
The joint Laplace-Stieltjes transform of the random vector
\begin{align*} (Q(t_{1}), Q(t_{2}), \ldots, Q(t_{m}), D(t_{1}), D(t_{2}), \ldots, D(t_{m})) \end{align*} is as follows: for $\boldsymbol{\alpha} := (\alpha_{1}, \alpha_{2}, \ldots, \alpha_{m}) \in \mathbb{R}_{+}^{m}$, $\boldsymbol{\beta} := (\beta_{1}, \beta_{2}, \ldots, \beta_{m}) \in \mathbb{R}_{+}^{m}$,
we have
\begin{eqnarray} \label{MainLSTresult1}
& & \mathbb{E}[e^{-\sum_{k=1}^{m}(\alpha_{k}Q(t_{k}) + \beta_{k}D(t_{k}))}] \\
&=& \prod_{k=1}^{m}\left[\prod_{n=1}^{\infty}\prod_{j_{k} = 1}^{n+1}\prod_{j_{k+1} = j_{k}}^{n+1} \cdots \prod_{j_{m} = j_{m-1}}^{n+1}\mathbb{E}\left[e^{-(\sum_{\ell = k}^{m}(\alpha_{\ell}(n - j_{\ell} + 1) + \beta_{\ell}(j_{\ell} - 1)))Y_{j_{k}, j_{k+1}, \ldots, j_{m}; n}(t_{k-1}, t_{k}]}\right]\right]. \nonumber
\end{eqnarray} where
\begin{align} \label{MainLSTresult11} &~~ \mathbb{E}\left[e^{-(\sum_{\ell = k}^{m}(\alpha_{\ell}(n - j_{\ell} + 1) + \beta_{\ell}(j_{\ell} - 1)))Y_{j_{k}, j_{k+1}, \ldots, j_{m}, n}(t_{k-1}, t_{k}]}\right] \\ &= e^{-(1 - e^{-\sum_{\ell = k}^{m}(\alpha_{\ell}(n - j_{\ell} + 1) + \beta_{\ell}(j_{\ell} - 1))})\int_{t_{k-1}}^{t_{k}}\mathcal{P}_{s}\left(\{B_{s} = n\} \cap \bigcap_{\ell = k}^{m}\{S_{j_{\ell} - 1:n}(s) \leq t_{\ell} - s, S_{j_{\ell}:n}(s) > t_{\ell} - s\}\right)\lambda(s)\mathrm{d}s}. \nonumber \end{align}  Furthermore, the auto-covariance functions of $\{Q(t); t \geq 0\}$ and $\{D(t); t \geq 0\}$ are as follows:  for each $t_{1}, t_{2}$ satisfying $0 < t_{1} < t_{2}$,
\begin{eqnarray} \label{MainLSTresult2}
\mathrm{Cov}(Q(t_{1}), Q(t_{2})) &=& \sum_{n=1}^{\infty}\sum_{j_{1} = 1}^{n+1}\sum_{j_{2} = j_{1}}^{n+1}(n - j_{1} + 1)(n - j_{2} + 1) \\
& & \times \int_{0}^{t}\mathcal{P}_{s}\left(\{B_{s} = n\} \cap \bigcap_{\ell = 1}^{2}\{S_{j_{\ell} - 1:n}(s) \leq t_{\ell} - s, S_{j_{\ell}:n}(s) > t_{\ell} - s\}\right)\lambda(s)\mathrm{d}s. \nonumber
\end{eqnarray}
and
\begin{eqnarray} \label{MainLSTresult3} \mathrm{Cov}(D(t_{1}), D(t_{2})) &=& \sum_{n=1}^{\infty}\sum_{j_{1} = 1}^{n+1}\sum_{j_{2} = j_{1}}^{n+1}(j_{1} - 1)( j_{2} - 1) \\ & & \times \int_{0}^{t}\mathcal{P}_{s}\left(\{B_{s} = n\} \cap \bigcap_{\ell = 1}^{2}\{S_{j_{\ell} - 1:n}(s) \leq t_{\ell} - s, S_{j_{\ell}:n}(s) > t_{\ell} - s\}\right)\lambda(s)\mathrm{d}s. \nonumber \end{eqnarray}
 \end{theorem}
\proof{Proof.} We begin by deriving both (\ref{MainLSTresult2}) and (\ref{MainLSTresult3}).  Considering first the random vector $(Q(t_{1}), Q(t_{2}))$, from Theorem~\ref{SACTQandD-FD} we see that for $(\alpha_{1}, \alpha_{2}) \in \mathbb{R}_{+}^{2}$,
\begin{eqnarray*}
& & \alpha_{1}Q(t_{1}) + \alpha_{2}Q(t_{2}) \\ &=& \alpha_{1}\sum_{n=1}^{\infty}\sum_{j_{1} = 1}^{n+1}\sum_{j_{2} = j_{1}}^{n+1}(n - j_{1} + 1)Y_{j_{1}, j_{2}; n}(0,t_{1}] \\ &+& \alpha_{2}\left[\sum_{n=1}^{\infty}\sum_{j_{1} = 1}^{n+1}\sum_{j_{2} = j_{1}}^{n+1}(n - j_{2} + 1)Y_{j_{1}, j_{2};n}(0,t_{1}] + \sum_{n=1}^{\infty}\sum_{j_{2} = 1}^{n+1}(n - j_{2} + 1)Y_{j_{2};n}(t_{1}, t_{2}]\right] \\ &=& \sum_{n=1}^{\infty}\sum_{j_{1} = 1}^{n+1}\sum_{j_{2} = j_{1}}^{n+1}(\alpha_{1}(n - j_{1} + 1) + \alpha_{2}(n - j_{2} + 1))Y_{j_{1}, j_{2};n}(0,t_{1}] + \sum_{n=1}^{\infty}\sum_{j_{2} = 1}^{n+1}\alpha_{2}(n - j_{2} + 1)Y_{j_{2};n}(t_{1}, t_{2}].
 \end{eqnarray*} Moreover, for $(\beta_{1}, \beta_{2}) \in \mathbb{R}_{+}^{2}$,
\begin{eqnarray*} & & \beta_{1}D(t_{1}) + \beta_{2}D(t_{2}) \\ &=& \beta_{1}\sum_{n=1}^{\infty}\sum_{j_{1} = 1}^{n+1}\sum_{j_{2} = j_{1}}^{n+1}(j_{1} - 1)Y_{j_{1}, j_{2}; n}(0,t_{1}] \\ &+& \beta_{2}\left[\sum_{n=1}^{\infty}\sum_{j_{1} = 1}^{n+1}\sum_{j_{2} = j_{1}}^{n+1}(j_{2} - 1)Y_{j_{1}, j_{2};n}(0,t_{1}] + \sum_{n=1}^{\infty}\sum_{j_{2} = 1}^{n+1}(j_{2} - 1)Y_{j_{2};n}(t_{1}, t_{2}]\right] \\ &=& \sum_{n=1}^{\infty}\sum_{j_{1} = 1}^{n+1}\sum_{j_{2} = j_{1}}^{n+1}(\beta_{1}(j_{1} - 1) + \beta_{2}(j_{2} - 1))Y_{j_{1}, j_{2};n}(0,t_{1}] + \sum_{n=1}^{\infty}\sum_{j_{2} = 1}^{n+1}\beta_{2}(j_{2} - 1)Y_{j_{2};n}(t_{1}, t_{2}].  \end{eqnarray*}
These representations for $Q(t)$, $Q(t+u)$, $D(t)$, and $D(t+u)$ can be used to derive the auto-covariance functions.  Indeed,
\begin{eqnarray*}
& & \mathrm{Cov}(Q(t_{1}), Q(t_{2})) \\ &=& \mathrm{Cov}\left(\sum_{n=1}^{\infty}\sum_{j_{1} = 1}^{n+1}\sum_{j_{2} = j_{1}}^{n+1}(n - j_{1} + 1)Y_{j_{1}, j_{2}; n}(0,t_{1}], \sum_{n=1}^{\infty}\sum_{j_{1} = 1}^{n+1}\sum_{j_{2} = j_{1}}^{n+1}(n - j_{2} + 1)Y_{j_{1}, j_{2};n}(0,t_{1}]\right) \\
&=& \sum_{n_{1}=1}^{\infty}\sum_{j_{1} = 1}^{n_{1}+1}\sum_{j_{2} = j_{1}}^{n_{1}+1}\sum_{n_{2} = 1}^{\infty}\sum_{k_{1} = 1}^{n_{2}+1}\sum_{k_{2} = k_{1}}^{n_{2}+1}(n_{2} - j_{1} + 1)(n_{2} - k_{2} + 1) \mathrm{Cov}(Y_{j_{1}, j_{2}; n_{1}}(0,t_{1}], Y_{k_{1}, k_{2};n_{2}}(0,t_{1}]) \\
&=& \sum_{n=1}^{\infty}\sum_{j_{1} = 1}^{n+1}\sum_{j_{2} = j_{1}}^{n+1}(n - j_{1} + 1)(n - j_{2} + 1) \mathrm{Var}(Y_{j_{1}, j_{2}; n}(0,t_{1}]) \end{eqnarray*}
which proves (\ref{MainLSTresult2}) since
\begin{align*}&~~ \mathrm{Var}(Y_{j_{1}, j_{2}; n}(0,t_{1}]) \\ &= \int_{0}^{t_{1}}\mathcal{P}_{s}\left(\{B_{s} = n\} \cap \bigcap_{\ell = 1}^{2}\{S_{j_{\ell} - 1:n}(s) \leq t_{\ell} - s, S_{j_{\ell}:n}(s) > t_{1} - s\}\right)\lambda(s)\mathrm{d}s \end{align*}  and a similar argument can be used to establish (\ref{MainLSTresult3}).

It remains to prove (\ref{MainLSTresult1}).  Given any collection of real numbers $0 = t_{0} < t_{1} < t_{2} < t_{3} < \cdots < t_{m-1} < t_{m}$, we have that for $(\alpha_{1}, \alpha_{2}, \ldots, \alpha_{m-1}, \alpha_{m}) \in \mathbb{R}_{+}^{m}$, $(\beta_{1}, \beta_{2}, \ldots, \beta_{m}) \in \mathbb{R}_{+}^{m}$,
\begin{eqnarray*}
& & \sum_{i=1}^{m}\alpha_{i}Q(t_{i}) + \sum_{i=1}^{m}\beta_{i}D(t_{i}) \\ &=& \sum_{k=1}^{m}\left[\sum_{n=1}^{\infty}\sum_{j_{k} = 1}^{n+1}\sum_{j_{k+1} = j_{k}}^{n+1} \cdots \sum_{j_{m} = j_{m-1}}^{n+1}\left[\sum_{\ell = k}^{m}(\alpha_{\ell}(n - j_{\ell} + 1) + \beta_{\ell}(j_{\ell} - 1))\right]Y_{j_{k}, j_{k + 1}, \ldots, j_{m}; n}(t_{k-1}, t_{k}]\right].
 \end{eqnarray*}
The representation from Theorem~\ref{SACTQandD-FD} shows that $\sum_{k=1}^{m}(\alpha_{k}Q(t_{k}) + \beta_{k}Q(t_{k}))$ can be expressed as a finite sum of independent, scaled Poisson random variables.  Further exploitation of this observation gives
\begin{eqnarray*} & & \mathbb{E}[e^{-\sum_{i=1}^{m}(\alpha_{i}Q(t_{i}) + \beta_{i}D(t_{i}))}] \\ &=& \prod_{k=1}^{m}\left[\prod_{n=1}^{\infty}\prod_{j_{k} = 1}^{n+1}\prod_{j_{k+1} = j_{k}}^{n+1} \cdots \prod_{j_{m} = j_{m-1}}^{n+1}\mathbb{E}\left[e^{-(\sum_{\ell = k}^{m}(\alpha_{\ell}(n - j_{\ell} + 1) + \beta_{\ell}(j_{\ell} - 1)))Y_{j_{k}, j_{k+1}, \ldots, j_{m}, n}(t_{k-1}, t_{k}]}\right]\right] \end{eqnarray*} which establishes (\ref{MainLSTresult1}), as clearly
\begin{align*} &~~ \mathbb{E}\left[e^{-(\sum_{\ell = k}^{m}(\alpha_{\ell}(n - j_{\ell} + 1) + \beta_{\ell}(j_{\ell} - 1)))Y_{j_{k}, j_{k+1}, \ldots, j_{m}, n}(t_{k-1}, t_{k}]}\right] \\ &= e^{-(1 - e^{-\sum_{\ell = k}^{m}(\alpha_{\ell}(n - j_{\ell} + 1) + \beta_{\ell}(j_{\ell} - 1))})\int_{t_{k-1}}^{t_{k}}\mathcal{P}_{s}\left(\{B_{s} = n\} \cap \bigcap_{\ell = k}^{m}\{S_{j_{\ell} - 1:n}(s) \leq t_{\ell} - s, S_{j_{\ell}:n}(s) > t_{\ell} - s\}\right)\lambda(s)\mathrm{d}s} \end{align*} due to $Y_{j_{k}, j_{k+1}, \ldots, j_{m}, n}(t_{k-1}, t_{k}]$ being a Poisson random variable.  This completes the proof of Theorem \ref{MainLSTresult}.  \hfill\Halmos \endproof

While the minimal assumptions on the arrival rates of batches and the services needed to achieve Theorem \ref{MainLSTresult} provide a broad generality for both the statement and its proof techniques, we recognize that it may occasionally be desirable to sacrifice the generality of our setting to achieve very simple results. If we consider specific settings, we are indeed able to derive simplified expressions. For example, in the case of stationary exponential service we can cleanly relate the auto-covariance and the variance.

\begin{proposition}
If the service is exponentially distributed at rate $\mu > 0$, the auto-covariance of the queue length is such that
\begin{align}
\Cov{Q(t),Q(t+\delta)} = \Var{Q(t)} e^{-\mu \delta}
,
\end{align}
for $t, \delta \geq 0$.
\end{proposition}
\proof{Proof.}
Since the queue length at time $t + \delta$ can be written as the queue length at $t$ plus the number of arrivals in $[t,t+\delta)$ and less the number of departures in $[t, t+\delta)$, i.e.
\begin{align*}
Q(t + \delta)
=
Q(t)
+
\sum_{i=1}^{A(t+\delta)} B_i - \sum_{i=1}^{A(t)} B_i
-
D(t+\delta) + D(t)
,
\end{align*}
we can decompose the auto-covariance accordingly. That is, by the definition of covariance we have that
\begin{align*}
\Cov{Q(t),Q(t+\delta)}
&=
\E{Q(t)Q(t+\delta)}
-
\E{Q(t)}\E{Q(t+\delta)}
\\
&=
\E{Q(t)\left(
Q(t)
+
\sum_{i=1}^{A(t+\delta)} B_i - \sum_{i=1}^{A(t)} B_i
-
D(t+\delta) + D(t)
\right)}
\\
&
\quad
-
\E{Q(t)}
\left(
\E{Q(t) }
+
\E{\sum_{i=1}^{A(t+\delta)} B_i - \sum_{i=1}^{A(t)} B_i }
-
\E{D(t+\delta) - D(t)}
\right)
.
\end{align*}
Since both the future of the arrival process and the sequence of batch sizes are independent from the history of queue, these terms cancel one another. With the linearity of expectation, this then simplifies to
 \begin{align*}
\Cov{Q(t),Q(t+\delta)}
&=
\Var{Q(t)}
-
\E{Q(t)(D(t+\delta) - D(t)}
+
\E{Q(t)}
\E{D(t+\delta) - D(t)}
.
\end{align*}
Given the queue length at time $t$, the number of departures on the interval $[t, t+\delta)$ can be written as a sum over all services that were completed. Using the memoryless-ness of exponential service, this means that
$$
D(t+\delta) - D(t)
=
\sum_{j=1}^{Q(t)}
\mathbf{1}\{S_j < \delta\}
,
$$
where $S_j \sim \mathrm{Exp}(\mu)$ are mutually independent and also independent of $Q(t)$. Through conditional expectation, we can also observe that
$$
\E{ Q(t) \sum_{j=1}^{Q(t)}
\mathbf{1}\{S_j < \delta\}}
=
\E{ Q(t) \sum_{j=1}^{Q(t)} \E{
\mathbf{1}\{S_j < \delta\} \mid Q(t) }}
=
\E{Q(t)^2} \PP{S_1 < \delta}
.
$$
Using this observation and the analogous result for the mean number of departures, we can further simplify the auto-covariance to
 \begin{align*}
\Cov{Q(t),Q(t+\delta)}
&=
\Var{Q(t)}
-
\E{Q(t)^2} \PP{S_1 < \delta}
+
\E{Q(t)}^2 \PP{S_1 < \delta}
\\
&=
\Var{Q(t)}
-
\Var{Q(t)}(1 - e^{-\mu \delta})
\\
&
=
\Var{Q(t)}e^{-\mu \delta}
,
\end{align*}
which completes the proof.
\hfill\Halmos \endproof

Although it is commonplace and thus expected for the exponential distribution to yield great simplicity, we can also find reduced expressions while maintaining some broader generality. In our next result, we find that under the assumptions where, for a batch that arrives at time $s$, all services within that batch are i.i.d. with CDF $F_{s}$, as well as independent of the batch size $B_{s}$, the finite-dimensional distributions of both $\{Q(t); t \geq 0\}$ and $\{D(t); t \geq 0\}$ simplify considerably from Theorem~\ref{MainLSTresult}.

\begin{theorem} \label{MainLSTiid} Suppose that within each batch arriving at time $s$, the amounts of work within that batch are i.i.d. with cumulative distribution function $F_{s}$.  Then the joint Laplace-Stieltjes transform of the random vector
\begin{align*} (Q(t_{1}), Q(t_{2}), \ldots, Q(t_{m}), D(t_{1}), D(t_{2}), \ldots, D(t_{m})) \end{align*} is as follows: for $\boldsymbol{\alpha} := (\alpha_{1}, \alpha_{2}, \ldots, \alpha_{m}) \in \mathbb{R}_{+}^{m}$, $\boldsymbol{\beta} := (\beta_{1}, \beta_{2}, \ldots, \beta_{m}) \in \mathbb{R}_{+}^{m}$, we have
\begin{eqnarray*} \mathbb{E}[e^{-\sum_{k=1}^{m}(\alpha_{k}Q(t_{k}) + \beta_{k}D(t_{k}))}] = e^{-\sum_{k=1}^{m}\sum_{b=1}^{\infty}\int_{t_{k-1}}^{t_{k}}\gamma_{k,b,\boldsymbol{\alpha}, \boldsymbol{\beta}}^{(m)}(s)\mathcal{P}_{s}(B_{s} = b)\lambda(s)\mathrm{d}s}.
\end{eqnarray*}
where for each $k \in \{1,2,\ldots,m\}$, each integer $b \geq 1$, and each $s \in [t_{k-1}, t_{k})$, we have
\begin{small}
\begin{align*} \gamma_{k,b,\boldsymbol{\alpha}, \boldsymbol{\beta}}^{(m)}(s) := 1 - e^{-b\sum_{x=k}^{m}\beta_{x}}\left[1 - \sum_{\ell = k+1}^{m}(1 - e^{-\sum_{x = k}^{\ell - 1}(\alpha_{x} - \beta_{x})})F_{s}(t_{\ell - 1} - s, t_{\ell} - s] - (1 - e^{-\sum_{x = k}^{m}(\alpha_{x} - \beta_{x})})\overline{F}_{s}(t_{m} - s)\right]^{b}. \end{align*} \end{small}
\end{theorem}

\proof{Proof.} Our objective now is to simplify the Laplace-Stieltjes transform found in (\ref{MainLSTresult11}).  Using standard properties of order statistics associated with i.i.d. random variables, we find that for each $s \in (t_{k-1}, t_{k}]$,
\begin{eqnarray} \label{MainLSTiidEq1}
& & \mathcal{P}_{s}\left(\bigcap_{\ell = k}^{m}\{S_{j_{\ell} - 1:n}(s) \leq t_{\ell} - s, S_{j_{\ell}:n}(s) > t_{\ell} - s\}\right) \\ &=& \frac{n!}{(j_{k} - 1)!(j_{k + 1} - j_{k})! \cdots (j_{m} - j_{m-1})!(n - j_{m} + 1)!} \nonumber
\\ &\times& F_{s}(t_{k}-s)^{j_{k} - 1}F_{s}(t_{k} - s, t_{k+1} - s]^{j_{k + 1} - j_{k}} \cdots F_{s}(t_{m-1} - s, t_{m} - s]^{j_{m} - j_{m-1}}\overline{F}_{s}(t_{m} - s)^{n - j_{m} + 1} \nonumber
\end{eqnarray}
where $F_{s}(s,t] := \mathcal{P}_{s}(s < S_{1}(s) \leq t)$, and $\overline{F}_{s}(t) = \mathcal{P}_{s}(S_{1}(s) > t)$.

Plugging this probability into (\ref{MainLSTresult1}), after plugging (\ref{MainLSTresult11}) into (\ref{MainLSTresult1}), and combining all exponential terms yields, within the exponent, the summation
\begin{eqnarray*}
\sum_{j_{k} = 1}^{n+1}\sum_{j_{k+1} = j_{k}}^{n+1}\cdots \sum_{j_{m} = j_{m-1}}^{n+1} & & (1 - e^{-\sum_{\ell = k}^{m}(\alpha_{\ell}(n - j_{\ell} + 1) + \beta_{\ell}(j_{\ell} - 1))})\frac{n!}{(j_{k} - 1)!(j_{k + 1} - j_{k})! \cdots (j_{m} - j_{m-1})!(n - j_{m} + 1)!}
\\ &\times& F_{s}(0, t_{k}-s]^{j_{k} - 1}F_{s}(t_{k} - s, t_{k+1} - s]^{j_{k + 1} - j_{k}} \cdots F_{s}(t_{m-1} - s, t_{m} - s]^{j_{m} - j_{m-1}} \\ &\times& \overline{F}_{s}(t_{m} - s)^{n - j_{m} + 1}
\end{eqnarray*} but this sum is simply
\begin{align*} &~~~ 1 - e^{-n\sum_{x = k}^{m}\beta_{x}}\left[F_{s}(0, t_{k} - s] + \sum_{\ell = k+1}^{m}e^{-\sum_{x = k}^{\ell - 1}(\alpha_{x} - \beta_{x})}F_{s}(t_{\ell - 1} - s, t_{\ell} - s] + e^{-\sum_{x = k}^{m}(\alpha_{x} - \beta_{x})}\overline{F}_{s}(t_{m} - s)\right]^{n} \\ &= 1 - e^{-n\sum_{x=k}^{m}\beta_{x}}\left[1 - \sum_{\ell = k+1}^{m}(1 - e^{-\sum_{x = k}^{\ell - 1}(\alpha_{x} - \beta_{x})})F_{s}(t_{\ell - 1} - s, t_{\ell} - s] - (1 - e^{-\sum_{x = k}^{m}(\alpha_{x} - \beta_{x})})\overline{F}_{s}(t_{m} - s)\right]^{n} \\ &= \gamma_{k,n,\boldsymbol{\alpha}, \boldsymbol{\beta}}^{(m)}(s). \end{align*}
Hence,
\begin{eqnarray*} \mathbb{E}[e^{-\sum_{k=1}^{m}(\alpha_{k}Q(t_{k}) + \beta_{k}D(t_{k}))}] = e^{-\sum_{k=1}^{m}\sum_{b=1}^{\infty}\int_{t_{k-1}}^{t_{k}}\gamma_{k,b,\boldsymbol{\alpha}, \boldsymbol{\beta}}^{(m)}(s)\mathcal{P}_{s}(B_{s} = b)\lambda(s)\mathrm{d}s} \end{eqnarray*} which proves the claim. \hfill\Halmos \endproof

\section{Almost Sure Convergence of Batch Scaling Limits}\label{limitSec}

Let us now take a deeper look at the scaling-limit theorems derived previously in the literature. Having studied the transforms of $Q(t)$, $D(t)$, and $W(t)$, we will turn to considering batch scaling limits of the system, in which the batch sizes grow large and the stochastic processes are normalized accordingly. Let us note that weak convergence of these limits will follow directly from the LST's provided in Theorems~\ref{JointLSTofQWD} (single dimensional) and~\ref{MainLSTresult} (finite-dimensional) {without invoking any results from the theory of Markov processes, as is done in Section 4 of} \citet{de2017shot}. Hence, here we focus instead on establishing strong convergence, which we believe to be the first almost sure convergence result in the batch scaling setting. We will restrict our attention to a scaling of a single dimensional object, but finite-dimensional results can be readily obtained using these same techniques.

Consider now a sequence of infinite-server queueing systems indexed by $m$, for each integer $m \geq 1$. Let the $m^\text{th}$ queueing system be an infinite-server queue with batch arrivals, whose queue-length process $\{Q(t); t \geq 0\}$ is of the form
\begin{align} \label{batch_queue}
Q_{m}(t) := \sum_{n=1}^{A(t)}\sum_{k=1}^{B_{n}^{(m)}}\mathbf{1}(S_{n,k} > t - T_{n})
,
\end{align}
where
\begin{align*}
B_{n}^{(m)} := \lceil m B_{n} \rceil
\end{align*}
represents the number of customers contain in the $n$th batch arrival to the system, and for each $n \geq 1$, the sequence $\{S_{n,k}\}_{k \geq 1}$ is an i.i.d. sequence of random variables, having CDF $F$.  We also associate with the $m^\text{th}$ queueing system the workload process $\{W_{m}(t); t \geq 0\}$, where
\begin{align*}
W_{m}(t) := \sum_{n=1}^{A(t)}\sum_{k=1}^{B_{n}^{(m)}}(S_{n,k} - (t - T_{n}))^{+}
\end{align*}
where $(x)^{+} := \max(x,0)$ for each $x \in \mathbb{R}$.

For now, we further assume that whenever a batch of customers arrive at time $t$, the amounts of work in the batch are i.i.d.~with CDF $F$, and independent of $B_{n}$ (i.e.~they are independent of the batch size itself).  Later, we will explain how this assumption can be relaxed to allow for services to be time-varying, but measurability issues arise if we try to stay in the most general setting discussed previously. We will maintain the Poisson assumptions on the arrival process $A(t)$ for thematic consistency, but readers can observe that $A(t)$ actually can be generalized to any simple point process. The following proof techniques will immediately carry over to such arrival processes.

\citet{de2017shot} established weak convergence of the batch scaled queue and workload to  the shot-noise processes $\{Z(t); t \geq 0\}$ and $\{Z^{w}(t); t \geq 0\}$, respectively, where $\{Z(t); t \geq 0\}$ and $\{Z^{w}(t); t \geq 0\}$ are defined as follows. For each $t \geq 0$,
\begin{align*}
Z(t) = \sum_{n=1}^{A(t)}B_{n}\overline{F}(t - T_{n})
,
\end{align*}
and
\begin{align*}
Z^{w}(t) = \sum_{n=1}^{A(t)}B_{n}\mathbb{E}[S_{1}]\overline{F}^{e}(t - T_{n})
,
\end{align*}
where $F^e$ is the residual or stationary excess distribution associated with $F$. We proceed with proving the strong convergence of these limits in the following two subsections, first for the queue length and then for the workload. We separate these results, as each requires a proof technique that would not suffice for the other.

\subsection{Limit of the Queue Length Process}

%We are interested in studying precisely in what way the processes $\{Q_{m}(t); t \geq 0\}$ can be approximated with a shot-noise process.
For each integer $m \geq 1$, and each real $t \geq 0$, let us define
\begin{align*}
\overline{Q}_{m}(t) := \frac{Q_{m}(t)}{m}.
\end{align*}
Our first result shows that for each $t \geq 0$, $\overline{Q}_{m}(t)$ converges almost surely to $Z(t)$.  One would think that this result should follow obviously from the strong law of large numbers, but after writing out $\overline{Q}_{m}(t)$, we find
\begin{eqnarray} \label{scaled_queue}
\overline{Q}_{m}(t) &=& \frac{Q_{m}(t)}{m} \\
&=& \frac{1}{m} \sum_{n=1}^{A(t)}\sum_{k=1}^{B_{n}^{(m)}}\mathbf{1}(S_{n,k} > t - T_{n}) \\
&=&  \sum_{n=1}^{A(t)}\frac{\lceil m B_{T_{n}} \rceil}{m}\frac{1}{B_{n}^{(m)}}\sum_{k=1}^{B_{n}^{(m)}}\mathbf{1}(T_{n} + S_{n,k} \leq t)
\end{eqnarray}
and the presence of the $T_{n}$ terms within the indicator function keeps us from applying the strong law directly, since $\{T_{n} + S_{n,k}\}_{k \geq 1}$ do not necessarily form an i.i.d.~sequence.  In fact, the indicator random variables $\{T_{n} + S_{n,k}\}_{k \geq 1}$ need not be independent nor identically distributed.  If they were independent and not identically distributed, we could easily use Kolmogorov's strong law of large numbers.  If they were identically distributed and not independent, it is possible to leverage standard correlation decay arguments.  Nevertheless, it is possible to get around this issue if one invokes the Glivenko-Cantelli Theorem, which we will now employ.
%A similar issue arises with the scaling-limit of the workload processes, and even though the Glivenko-Cantelli Theorem cannot be applied in that case, we can use a different approach that relies on the fact that the CDF of the stationary excess life $F^{e}$ associated with $F$ is continuous on $\mathbb{R}$.

\begin{theorem} \label{QueueWP1}
For each $t \geq 0$, we see that as $m \rightarrow \infty$,
\begin{align*}
\overline{Q}_{m}(t) \stackrel{\mathsf{a.s.}}{\longrightarrow} Z(t).
\end{align*}
\end{theorem}

\proof{Proof.}
Fix a real $t \geq 0$. Given an integer $n \geq 1$, we can see from the Glivenko-Cantelli Theorem that
\begin{align*}
\sup_{x \in \mathbb{R}}\left|F_{n,m}(x) - F(x)\right| \stackrel{a.s.}{\rightarrow} 0
\end{align*}
as $m \rightarrow \infty$, where $F_{n,m}$ is the empirical CDF associated with the random sample $S_{n,1}, S_{n,2}, \ldots, S_{n,m}$, i.e.
\begin{align*}
F_{n,m}(x) := \frac{1}{m}\sum_{k=1}^{m}\mathbf{1}(S_{n,k} \leq x).
\end{align*}
Due to this almost sure convergence result, we can find a null set $N_{n}$ such that for each $\omega \in N_{n}^{c}$,
\begin{align*}
 \lim_{m \rightarrow \infty} \sup_{x \in \mathbb{R}}\left|\frac{1}{m}\sum_{k=1}^{m}\mathbf{1}(S_{n,k}(\omega) \leq x) - F(x)\right| = 0.
 \end{align*}
 Next, observe that since $n$ is an arbitrarily chosen positive integer, by defining
\begin{align*}
N := \bigcup_{n=1}^{\infty}N_{n}
,
\end{align*}
we see from Boole's inequality that $N$ is also a null set, and for each $\omega \in N^{c}$,
\begin{align*}
\lim_{m \rightarrow \infty} \sup_{x \in \mathbb{R}}\left|\frac{1}{m}\sum_{k=1}^{m}\mathbf{1}(S_{n,k}(\omega) \leq x) - F(x)\right| = 0
\end{align*}
for each $x \in \mathbb{R}$, and for each integer $n \geq 1$.  Finally, letting $M$ be a null set where for each $\omega \in M^{c}$,
\begin{align*}
\lim_{n \rightarrow \infty}T_{n}(\omega) = \infty
,
\end{align*}
we conclude that for each $\omega \in N^{c} \cap M^{c}$,
\begin{align*}
\lim_{m \rightarrow \infty}\frac{1}{m}\sum_{n=1}^{A(t)}\sum_{k=1}^{B_{n}^{(m)}(\omega)}\mathbf{1}(S_{n,k}(\omega) > t - T_{n}(\omega)) &= \sum_{n=1}^{A(t)}\lim_{m \rightarrow \infty}\frac{\lceil m B_{n}(\omega) \rceil}{m}\frac{1}{B_{n}^{(m)}}\sum_{k=1}^{B_{n}^{(m)}}\mathbf{1}(S_{n,k}(\omega) > t - T_{n}(\omega) \\ &= \sum_{n=1}^{A(t)}B_{n}\overline{F}(t - T_{n}(\omega))) \end{align*}
which proves the claim.
\hfill\Halmos \endproof

Let us note that although we have restricted our attention to the queue length process here, the corresponding limit of the departure follows immediately from Theorem~\ref{QueueWP1}.

\subsection{Limit of the Workload Process}

In our next result, we show that $\{\overline{W}_{m}(t); t \geq 0\}$ exhibits a similar type of convergence. For each integer $m \geq 1$, and each real $t \geq 0$, we define
\begin{align*}
\overline{W}_{m}(t) := \frac{W_{m}(t)}{m}.
\end{align*}
One might also think that the almost sure convergence of $\overline{W}_{m}(t)$ to $Z^w(t)$ should follow obviously from the strong law of large numbers, but after writing out $\overline{W}_{m}(t)$, we find
\begin{align*}
W_{m}(t) := \sum_{n=1}^{A(t)}\sum_{k=1}^{B_{n}^{(m)}}(S_{n,k} - (t - T_{n}))^{+}
\end{align*}
and the presence of the $T_{n}$ terms within the indicator function again keeps us from applying the strong law directly, since $\{T_{n} + S_{n,k}\}_{k \geq 1}$ do not necessarily form an i.i.d.~sequence. By comparison to the batch scaling of the queue, the Glivenko-Cantelli Theorem cannot be applied in this case. However, we can use a different approach that now relies on the fact that the CDF of the stationary excess distribution $F^{e}$ is continuous on $\mathbb{R}$, which need not be the case for $F$ itself.

\begin{theorem} \label{WorkloadWP1}
For each $t \geq 0$, we see that as $m \rightarrow \infty$,
\begin{align*}
\overline{W}_{m}(t) \stackrel{\mathsf{a.s.}}{\longrightarrow} Z^{w}(t) .
\end{align*}
%where $\{Z^{w}(t); t \geq 0\}$ is a shot-noise process of the form
%\begin{align*}
%Z^{w}(t) := \sum_{n=1}^{\infty}B_{n}\mathbb{E}[S_{n,1}]\overline{F}^{e}(t - T_{n}).
%\end{align*}
\end{theorem}

\proof{Proof.}
Fix a real $t \geq 0$.  Given an integer $n \geq 1$, observe that for each rational $r \in \mathbb{Q}$,
\begin{align*}
\frac{1}{m}\sum_{k=1}^{m}(S_{n,k} - r)^{+} \stackrel{\mathsf{a.s.}}{\longrightarrow}\mathbb{E}[(S_{n,1} - r)^{+}].
\end{align*}
Due to the fact that the rationals form a countable set, we can see from the ordinary Strong Law of Large Numbers that there exists a null set $N_{n}$ such that for each $\omega \in N_{n}^{c}$,
\begin{align*}
\lim_{m \rightarrow \infty}\frac{1}{m}\sum_{k=1}^{m}(S_{n,k} - r)^{+} = \mathbb{E}[(S_{n,1} - r)^{+}]
\end{align*}
for each rational $r \in \mathbb{Q}$.  The same can be said for each integer $n \geq 1$:  setting then
\begin{align*}
N := \bigcup_{n=1}^{\infty}N_{n}
\end{align*}
we see that $N$ is a null set, and for each $\omega \in N^{c}$,
\begin{align*}
\lim_{m \rightarrow \infty}\frac{1}{m}\sum_{k=1}^{m}(S_{n,k} - r)^{+} = \mathbb{E}[(S_{n,1} - r)^{+}]
\end{align*}
for each rational $r \in \mathbb{Q}$, and each integer $n \geq 1$.

Given $\omega \in N^{c}$, we can see that for each integer $n \geq 1$ satisfying $T_{n}(\omega) \leq t$, we can find rational numbers $a_{n}, b_{n}$ satisfying $a_{n} < t - T_{n}(\omega) < b_{n}$, where
\begin{align*}
\mathbb{E}[S_{n,1}]\overline{F}^{e}(b_{n}) &= \lim_{m \rightarrow \infty}\frac{1}{m}\sum_{k=1}^{m}(S_{n,k}(\omega) - b_{n})^{+} \\ &\leq \liminf_{m \rightarrow \infty}\frac{1}{m}\sum_{k=1}^{m}(S_{n,k}(\omega) - (t - T_{n}(\omega))^{+} \\ &\leq \limsup_{m \rightarrow \infty}\frac{1}{m}\sum_{k=1}^{m}(S_{n,k}(\omega) - (t - T_{n}(\omega)))^{+} \\ &\leq \lim_{m \rightarrow \infty}\frac{1}{m}\sum_{k=1}^{m}(S_{n,k}(\omega) - a_{n})^{+} = \mathbb{E}[S_{n,1}]\overline{F}^{e}(a_{n}).
\end{align*}
Since $a_{n}$ and $b_{n}$ were chosen arbitrarily, we can let $a_{n}$ approach $t - T_{n}(\omega)$ from below, and let $b_{n}$ approach $t - T_{n}(\omega)$ from above to conclude that
\begin{align*}
\lim_{m \rightarrow \infty}\frac{1}{m}\sum_{k=1}^{m}(S_{n,k}(\omega) - (t - T_{n}(\omega)))^{+} = \mathbb{E}[S_{n,1}]\overline{F}^{e}(t - T_{n})
\end{align*}
where we also made use of the continuity of $\overline{F}^{e}$.  Once we intersect $N^{c}$ with the set on which $T_{n} \rightarrow \infty$ as $n \rightarrow \infty$ (assured for Poisson processes or any simple point process {that does not explode in finite time}), we conclude that
\begin{align*}
\overline{W}_{m}(t) \stackrel{\mathsf{a.s.}}{\longrightarrow} Z^{w}(t)
\end{align*}
as $m \rightarrow \infty$.
\hfill\Halmos \endproof

\noindent \textbf{Remark} Readers should observe that the proof we used to establish Theorem \ref{WorkloadWP1} cannot be used to prove Theorem \ref{QueueWP1}, unless we further assume $F$ does not have any jumps.

 What is especially notable about the batch scaling limit is that even though both $\{\overline{Q}_{m}(t); t \geq 0\}$ and $\{\overline{W}_{m}(t); t \geq 0\}$ converge to shot-noise processes, the two shot-noise processes have different decay patterns:  the decay pattern associated with the scaling-limit of the queue-lengths is $\overline{F}(t)$, while the decay pattern associated with the scaling-limit of the workload processes is $\mathbb{E}[S_{1}]\overline{F}^{e}(t)$.  We can also see that, when $F$ corresponds to the CDF of an exponential random variable with rate $\mu$, the two scaling limits actually coincide if we further multiply $\overline{W}_{m}(t)$ by $\mu$.

As we have discussed, this section has stood apart from the rest of the paper for its lack of time-variation. The following corollary shows that we can relax both Theorems \ref{QueueWP1} and \ref{WorkloadWP1} to some extent, and allow for service time distributions to be time-varying in a certain manner that should cover most practical applications.

\begin{corollary} \label{TimeVaryingWP1}
Suppose there exists a strictly increasing sequence of real numbers $\{s_{k}\}_{k \geq 0}$ satisyfing $s_{k} \rightarrow \infty$ as $k \rightarrow \infty$, $s_{0} = 0$, and for each integer $k \geq 1$, whenever a batch of customers arrive at a time $s \in [s_{k-1}, s_{k})$, each customer in that batch possesses an amount of work with CDF $F_{k}$ that is independent of all other works in the batch, as well as the size of the batch.  Under these conditions, it follows that for each $t \geq 0$,
\begin{align*}
\overline{Q}_{m}(t) \stackrel{\mathsf{a.s.}}{\longrightarrow} Z(t), ~~~~~ \overline{W}_{m}(t) \stackrel{\mathsf{a.s.}}{\longrightarrow} Z^{w}(t)
\end{align*}
as $m \rightarrow \infty$.
\end{corollary}

\proof{Proof.}
The proofs of both Theorems \ref{QueueWP1} and \ref{WorkloadWP1} can be modified in an obvious way, by relating to each interval $[s_{j-1}, s_{j})$ a doubly-indexed collection of i.i.d. random variables $\{S_{n,k}^{(j)}\}_{n \geq 1, k \geq 1}$ having CDF $F_{j}$, then using this doubly-indexed sequence to `feed' the $m$th queueing system with service times.
\hfill\Halmos \endproof

Readers should note that our proofs of Theorems \ref{QueueWP1} and \ref{WorkloadWP1} will no longer be valid if we let $F_{t}$ vary too much with respect to $t$.  If we try to apply that approach to each of these settings, we run into issues with constructing a suitable null set because our construction will lead to an uncountable number of null sets. Addressing this is thus left as an interesting direction of future work.

\section{Conclusion}\label{concSec}

%Infinite server queues are vital to the analysis of more complicated queueing systems.  In this paper, we prove strong convergence results for the batch scaling limit of \citet{daw2020non}.  \citet{daw2018distributions, eick1993physics}

A fundamental takeaway of our work is expanded understanding of the remarkable tractability granted to queues with Poisson arrivals and infinitely many servers. Even though we have allowed the arrival rates, batch size distributions, and service duration distributions to be broadly time-varying (and furthermore allowed arbitrary dependence among service durations within a batch, these two assumptions on the arrival process and the number of servers alone provide a decomposition into the sum of independent Poisson random variables. In Theorem~\ref{SACTQandD} we proved this representation for a single point in time, and in Theorem~\ref{SACTQandD-FD} we extended to multiple epochs for perspectives finite-dimensional distributions.

While this decomposition is interesting in its own right, we have shown here that it is also quite valuable methodologically. That is, in this paper we have also analyzed various transforms and performance metrics of the stochastic process, and the proof techniques we have used were built upon the decomposition we have found. For example, in Theorem~\ref{MainLSTresult} we used the multi-point decomposition to find the joint LST of the queue and departure processes across multiple epochs, in addition to the auto-covariances of the queue and departure processes. We hope that these results and, perhaps even more so, the techniques used to find them will be of use in our motivating applications, such as the management of autonomous vehicles, multi-server jobs, and microservices.

At the core of this paper, we have been interested in generalizing prior results for batch arrival infinite server queues. In Section~\ref{limitSec}, we extended the batch scaling limits first shown by \citet{de2017shot}, which connected the batch arrival queues with shot-noise processes. Prior results established weak convergence of these limits, and here we have generalized this to hold almost surely. We prove these limits for both the queue length and workload processes, and the proof techniques differ for each. For the queue, Theorem~\ref{QueueWP1} uses the Glivenko-Cantelli theorem, whereas Theorem~\ref{WorkloadWP1} leverages the continuity of the stationary excess distribution. As we have noted, these proof techniques actually do not require the Poisson arrivals assumption, and thus Theorems~\ref{QueueWP1} and~\ref{WorkloadWP1} will hold for any simple point process.

Naturally, interesting directions of future work lie in relaxing each of these two fundamental assumptions, the Poisson arrivals and the infinitely many servers, that are the pillars of the model's tractability. For example, we are interested in aiming to extend this method of decomposition to batch arrival queues with arrival processes that are not Poisson. For arrival processes that are not Poisson processes but may be closely related, such as Cox processes, it may be promising to try to leverage this near-Poisson-ness, and thus gain insight into the distribution of the batch arrival queues overall. Then, on the other hand, we have recently studied the staffing problem for the batch arrival multi-server queue, and this has invoked batch scaling limits of the queue \citep{daw2021staff}. Thus, we are also quite interested to see what the almost sure convergence from Theorems~\ref{QueueWP1} and~\ref{WorkloadWP1} can add to this analysis.

\section*{Acknowledgements}

This research was initiated when Andrew Daw was a doctoral student at Cornell, at which time he was supported by the National Science Foundation through a Graduate Research Fellowship under grant DGE-1650441.

\bibliographystyle{plainnat}
\bibliography{BatchNS}

\end{document}